\documentclass{amsart} 

\usepackage[french]{babel}
\usepackage[latin1]{inputenc}

\usepackage{amsmath}
\usepackage{amsthm}
\usepackage{amssymb}
\usepackage{tikz}

\usepackage{array}
\usepackage{xy}
\xyoption{all}
\usepackage{mathtools}

\usepackage{imakeidx} 
\usepackage{appendix}
\usepackage{tikz-cd}
\usepackage{verbatim}
\usepackage{enumitem}
\usepackage{multicol}

\usepackage[backref=page]{hyperref}
\usepackage{cleveref}

 \hypersetup{colorlinks=true,linkcolor=blue,anchorcolor=blue,citecolor=blue}

\usepackage{adjustbox}

\usepackage{mathtools}

\usepackage{color}

\newtheorem{thm}{Th\'eor\`eme}[section]
\newtheorem{prop}[thm]{Proposition}
\newtheorem{lemma}[thm]{Lemme}

\newtheorem{lemang}[thm]{Lemma}

\theoremstyle{definition}

\theoremstyle{remark}

\newtheorem{rmk}[thm]{Remarque}
\newtheorem{remark}[thm]{Remark}
\numberwithin{equation}{section}

\newcommand{\Gal}{{\rm Gal}}
\newcommand{\Br}{{\rm Br}}
\newcommand{\Q}{\mathbb Q}
\newcommand{\R}{\mathbb R}
\newcommand{\F}{\mathbb F}
\newcommand{\C}{\mathbb C}
\newcommand{\Z}{\mathbb Z}
\newcommand{\g}{\mathfrak g}

\renewcommand{\P}{\mathbb P}
\newcommand{\Spec}{\operatorname{Spec}}

\newcommand{\A}{\mathbb A}

\newcommand{\cl}{\overline}

\renewcommand{\phi}{\varphi}

\title[Intersections de deux quadriques]{Retour sur l'arithm\'etique des intersections  de deux quadriques, \author{Jean-Louis Colliot-Th\'el\`ene}
\\avec un appendice par A. Kuznetsov}

\address{Universit\'e Paris-Saclay, CNRS, Laboratoire de math\'ematiques d'Orsay, 91405, Orsay, France.}
\email{jean-louis.colliot-thelene@universite-paris-saclay.fr}

\address{
Steklov Mathematical Institute of Russian Academy of Sciences, Moscow, Russia
Laboratory of Algebraic Geometry, NRU HSE, Moscow, Russia}
\email{akuznet@mi-ras.ru}

\date{28 novembre 2023 :  J. reine angew. Math. DOI 10.1515/crelle-2023-0081}

\begin{document}
	\maketitle
	\hypersetup{backref=true}

\begin{abstract} 
Soit $k$ un corps $p$-adique. On montre que toute intersection de deux quadriques dans l'espace projectif $\P^4_{k}$
contient un point sur une extension quadratique, ce qui g\'en\'eralise un r\'esultat de Creutz et Viray
pour le cas lisse. La preuve utilise un th\'eor\`eme de Lichtenbaum sur les courbes de genre 1 sur un corps
$p$-adique.
 On d\'eduit de ce r\'esultat que toute intersection lisse de deux quadriques  $X \subset \P^5_{k}$ 
sur un corps de nombres $k$ poss\`ede un point sur une extension quadratique. 
On d\'eduit aussi de ce r\'esultat une d\'emonstration
 relativement courte  d'un th\'eor\`eme de Heath-Brown :
le principe de Hasse vaut pour les intersections compl\`etes lisses $X \subset \P^7_{k}$ de deux quadriques
sur un corps de nombres. On donne aussi une d\'emonstration alternative d'un principe de Hasse
pour certaines intersections de deux quadriques dans $\P^5_{k}$, d\^{u} \`a Iyer et Parimala.

Lichtenbaum proved that index and period coincide for a curve of genus one
over a $p$-adic field. Salberger proved that the Hasse principle holds for a
smooth complete intersection of two quadrics  $X \subset \P^n$
 over a number field, if $n \geq 5$ and $X$
 contains a conic.
 Building upon these two results, we extend recent results of Creutz and Viray
 (2021) on the existence of a quadratic point on intersections of two quadrics over $p$-adic fields
 and over number fields.  We then recover  Heath-Brown's theorem (2018)
  that the Hasse principle holds for smooth
complete intersections of two quadrics in $\P^7$. We also give an alternate proof of a theorem of
Iyer and Parimala (2022)  on the local-global principle in the case $n=5$.
\end{abstract}

\section{Introduction}\label{intro}

\def\Pic{{\rm Pic}}
\def\k{{\overline k}}
\def\C{{\overline C}}

\subsection{Le contexte}

Une  conjecture g\'en\'erale  \cite[Conjecture 14.1.2]{CTSk21}  postule que l'obstruction de Brauer-Manin est la seule obstruction au principe de Hasse pour les points rationnels des vari\'et\'es projectives, lisses, g\'eom\'etriquement
rationnellement connexes sur un corps  de nombres   $k$. Cette conjecture fut faite par Sansuc et moi  \cite{CTSa80} dans le cas des surfaces g\'eom\'etriquement rationnelles, et formul\'ee en dimension sup\'erieure  par  moi  en 2000. Lorsque la vari\'et\'e $X$ satisfait de plus que le quotient
 $\Br(X)/{\rm Im}(\Br(k))$   du groupe de Brauer  de $X$ par l'image
du groupe de Brauer du corps de base est nul, 
 la conjecture dit  que le principe de Hasse vaut pour les points rationnels de $X$ : si la vari\'et\'e $X$ a des points
 dans tous les compl\'et\'es $k_{v}$ de $k$, alors $X$ a un point rationnel.
 On renvoie le lecteur au rapport de Wittenberg \cite{Wi15} et \`a \cite[Chap. 14]{CTSk21}
 pour une description r\'ecente des recherches dans ce domaine.

Soit  d\'esormais  $X\subset   \P^n_{k}$ une {\it  intersection compl\`ete   lisse  de deux quadriques}.
Pour $n \geq 4$, c'est une vari\'et\'e g\'eom\'etri\-que\-ment rationnelle \cite[Prop. 2.2]{CTSaSD87}, a~fortori est-elle g\'eom\'etri\-quement rationnellement
connexe. Pour $n \geq 5$, on a de plus $\Br(X)/{\rm Im}(\Br(k))=0$.
On sait  montrer \cite[Thm. 4.2.1, Prop. 5.2.3]{Ha94}
 que si la conjecture g\'en\'erale vaut  pour un entier $m\geq 4$
et  toute  $X \subset \P^m_{k}$,
 alors elle vaut  pour  tout $n \geq m$ et  toute telle   $X \subset \P^n_{k}$. 
 Pour $m\geq 5$, ceci dit simplement
 que si le principe de Hasse vaut pour toute $X \subset \P^m_{k}$ avec un $m \geq 5$,
 alors il vaut pour toute $X \subset \P^n_{k}$ pour tout $n \geq m$.

En 1987, l'article  \cite{CTSaSD87} a \'etabli   
  la conjecture g\'en\'erale pour $n \geq 4$ sous l'hypoth\`ese que
$X$ contient un couple de droites   conjugu\'ees, et  pour $n\geq 8$ a \'etabli le 
principe de Hasse pour toute $X \subset \P^n_{k}$.

 En 1988 et 1989, on montra, par deux m\'ethodes distinctes \cite{CT88, Sal88, Sal89, SalSk91}, 
 que la conjecture g\'en\'erale vaut pour $X \subset \P^4_{k}$ 
 contenant une conique, ou, en d'autres
 termes, pour $X \subset \P^4_{k}$ admettant une fibration en coniques sur la droite projective.
 
En 1993, Salberger a \'etabli  la conjecture g\'en\'erale
  pour     $n \geq 5$,  lorsque $X$ contient   une conique \cite{Sal93} \cite[Prop. 5.2.6]{Ha94}. 
  
En 2006, O. Wittenberg \cite{Wi07}  a montr\'e comment sous la  combinaison de deux conjectures
classiques mais pour l'instant hors d'atteinte, \`a savoir la finitude des groupes de Tate-Shafarevich
des courbes elliptiques et la conjecture de   Schinzel, le principe de Hasse
vaut pour   $ X \subset \P^n_{k}$ et $n\geq 5$.

En 2017,  dans le cas analogue d'un corps global $k$ de caract\'eristique $p>2$, par  des m\'ethodes
g\'eom\'etriques,  Zhiyu Tian \cite{ZT17} a \'etabli le principe de Hasse pour $n \geq 5$.
 
En 2018, R. Heath-Brown \cite{HB18} a \'etabli le principe de Hasse 
pour  $n=7$  sur un corps de nombres. Par rapport \`a \cite{CTSaSD87},
le point nouveau  essentiel dans son travail
est la d\'emonstration que, sur un corps $p$-adique,
 toute intersection compl\`ete lisse de deux quadriques  dans $\P^7_{k}$  
  qui contient un point rationnel contient une conique.

Le pr\'esent article a \'et\'e suscit\'e par le travail de Heath-Brown et par deux articles r\'ecents.

D'une part B. Creutz et B. Viray  \cite{CV21} ont \'etudi\'e l'existence de points
quadratiques sur $X \subset \P^n_{k}$
pour $n \geq 4$ et $k$ local ou global. Sur un corps local, ils montrent qu'il existe
un point ferm\'e de degr\'e au plus 2. Sur un corps global, ils montrent
 que le pgcd des degr\'es des points ferm\'es (l'indice de $X$)  divise 2.

 D'autre part J. Iyer et R. Parimala \cite{IP22} ont \'etabli un cas particulier
 de la conjecture pour $X \subset \P^5_{k}$
 (th\'eor\`eme \ref{PariIyer1} ci-dessous).

Ces travaux am\`enent \`a s'interroger \`a nouveau
  sur les intersections de deux quadriques lisses $X \subset \P^n_{k}$
d\'efinies sur un corps local ou sur un corps global,
  sur l'existence   de droites d\'efinies sur le corps de base ou
 sur une extension quadratique, et sur l'existence de coniques,
 dans le but d'utiliser ces courbes
 pour \'etablir le principe de Hasse pour les points rationnels dans le cas global.

 \subsection{Principaux r\'esultats}

Le r\'esultat principal de l'article est une d\'emonstration nous semble-t-il
 conceptuelle des deux  th\'eor\`emes  suivants dus \`a Heath-Brown \cite[Thm. 1, Thm. 2]{HB18}.
 
 \begin{thm}\label{hbrloc}
Soient  $k$ un corps $p$-adique et $X \subset \P^7_{k}$ une intersection compl\`ete lisse
de deux quadriques poss\'edant un point rationnel. Dans le pinceau des formes quadratiques s'annulant sur $X$,
il existe des formes de rang 8 qui contiennent trois hyperboliques, i.e. qui s'annulent 
sur un espace lin\'eaire $\P^2_{k} \subset  \P^7_{k}$.
\end{thm}

\begin{thm}\label{hbrgl}
Soient $k$ un corps de nombres et  $X \subset \P^7_{k}$ une intersection compl\`ete lisse
de deux quadriques. Le principe de Hasse
vaut pour $X$ : si $X$ a des points rationnels dans tous les compl\'et\'es $k_{v}$
de $k$, alors $X$ a des points rationnels.
\end{thm}

Un th\'eor\`eme g\'en\'eral  sur les intersections lisses   de deux quadriques dans $\P^n_{k}$
pour $n \geq 5$  qui poss\`edent un point rationnel
\cite[Thm. 3.11]{CTSaSD87} donne alors l'approximation faible pour $X$,
que nous ne discutons donc pas dans le pr\'esent article.

\medskip

Nous donnons par ailleurs  une d\'emonstration originale d'un th\'eor\`eme r\'ecent
d'Iyer et Parimala :

 \begin{thm}\label{PariIyer1} 
Soient $k$ un corps de nombres et $X \subset \P^5_{k}$ une intersection compl\`ete
lisse de deux quadriques donn\'ee par un syst\`eme $f=0, g=0$.
Supposons que la courbe $C$ d'\'equation 
$y^2=-det(\lambda f + \mu g)$ a un point de degr\'e impair,
ce qui est le cas par exemple s'il existe une forme de rang 5 dans le pinceau.
Si, pour chaque place $v$, la vari\'et\'e $X$ contient une $k_{v}$-droite,
alors $X$ a un point rationnel.
\end{thm} 
C'est le th\'eor\`eme \ref{iyerparimala} (voir aussi le th\'eor\`eme \ref{autreIP}).

\medskip

Nous obtenons aussi les r\'esultats suivants.

\begin{thm}\label{quadlocal}
 Sur un corps $p$-adique $k$,  sur toute intersection  $X$
 de deux quadriques dans $\P^n_{k}$, $n \geq 4$,
 il existe un point dans une extension quadratique de $k$.
 \end{thm}
 C'est le th\'eor\`eme \ref{dansP4}.  Ceci g\'en\'eralise
   le th\'eor\`eme \cite[Thm. 1.2 (1)]{CV21} obtenu par
  Creutz et Viray  dans le cas lisse.
  
 \begin{thm}\label{quadglobal}
 Sur un corps de nombres,
   toute intersection compl\`ete lisse de deux quadriques dans $\P^n_{k}$
   pour $n \geq 5$ poss\`ede un point quadratique.
   \end{thm}
  C'est le th\'eor\`eme \ref{quadratiqueP5nombres}.  Ceci  donne une d\'emonstration
  inconditionnelle du th\'eor\`eme conditionnel \cite[Thm. 1.2 (4)]{CV21}.
 
 \medskip
 
 Dans la perspective de recherches ult\'erieures, on a syst\'ematiquement
 enregistr\'e un certain nombre de r\'esultats g\'en\'eraux
 sur les sous-espaces lin\'eaires des intersections de deux quadriques, et sur 
 les sous-espaces lin\'eaires contenus dans les quadriques contenant ces vari\'et\'es,
 sur divers corps de base (finis, locaux, globaux).

 \subsection{Structure de l'article}

 \medskip
  
 Au \S \ref{para2},  on fait un certain nombre de rappels sur l'alg\`ebre, la g\'eom\'etrie et l'arith\-m\'etique
 des  intersections compl\`etes lisses de deux quadriques $X \subset \P^n_{k}$, $n \geq 3$.
  
  On donne de nombreux rappels sur la g\'eom\'etrie des espaces lin\'eaires contenus
 dans une intersection compl\`ete lisse de deux quadriques $X \subset \P^n_{k}$, et sur la g\'eom\'etrie
 des vari\'et\'es param\'etrant un espace lin\'eaire contenu dans une quadrique
 contenant $X$.
  
  On rappelle  un th\'eor\`eme local-global de Salberger  \cite{Sal93}
 sur les intersections de deux quadriques $X \subset \P^n_{k}$ qui con\-tien\-nent une conique (th\'eor\`eme \ref{sal93} ci-dessous).
 Ce th\'eor\`eme joue ici  un r\^ole fondamental dans la d\'emonstration
 des th\'eor\`emes \ref{hbrgl} et \ref{PariIyer1}.
  
  \medskip
  
  Au \S \ref{para3},  on se place sur un corps $p$-adique $k$ et on s'int\'eresse
  \`a l'existence d'un point au plus quadratique sur une intersection
  de deux quadriques.
On commence par  le cas crucial des courbes $X \subset \P^3_{k}$
(donc des courbes de genre 1 dans le cas lisse).
   En utilisant des r\'esultats de Roquette et Lichtenbaum \cite{Li68, Li69}
  qui reposent sur le th\'eor\`eme de dualit\'e de Tate 
  pour les vari\'et\'es ab\'eliennes sur un corps $p$-adique,
  on \'etablit  :
  \begin{prop}\label{intP3intro}
     Soient $f(x_{0}, x_{1}, x_{2},x_{3} )$ et $g(x_{1}, x_{2},x_{3})$ deux formes
     quadratiques sur un corps $p$-adique. La $k$-vari\'et\'e $X \subset \P^3_{k}$
     d\'efinie  par $f=g=0$ poss\`ede un point dans une extension quadratique
     de $k$.
     \end{prop}
     C'est la proposition \ref{intersectquelconqueP3}.

 On revient sur des r\'esultats de Creutz et Viray.
  Ceux-ci ont montr\'e \cite[Thm. 1.2 (1)]{CV21}
 que   sur tout corps $p$-adique $k$
  toute intersection  compl\`ete {\it lisse} $X$
 de deux quadriques dans $\P^n_{k}$, $n \geq 4$, poss\`ede
  un point dans une extension quadratique.
On \'etablit et g\'en\'eralise  ce r\'esultat aux intersections quelconques.
 C'est le th\'eor\`eme \ref{dansP4}, qui r\'esulte facilement de la proposition 
 \ref{intersectquelconqueP3} dans le cas $n=3$.
 Cette m\'ethode  permet d'\'eviter les 
d\'elicats calculs locaux, en particulier 2-adiques,
 de \cite[\S 2, Proof of Thm. 2.1]{CV21} et de l'article \cite{HB18}.
 
Le th\'eor\`eme de Lichtenbaum
\'etait mentionn\'e dans dans  \cite[Prop. 4.5]{CV21},
pour une d\'emonstration alternative d'un \'enonc\'e
 un peu plus faible dans le cas $n=4$.

La proposition \ref{intP3intro}
forme  la base de notre m\'ethode pour \'etablir, au \S \ref{para7},
  le th\'eor\`eme  local de Heath-Brown (th\'eor\`eme \ref{hbrloc}) dans $\P^7_{k}$. 

    \medskip
    
    Dans les paragraphes suivants, pour $4 \leq n \leq 7$,
     pour $k$ fini,   local,  global, et $X$
   intersection compl\`ete lisse   de deux quadriques dans $\P^n_{k}$ donn\'ee par 
   un syst\`eme $f=g=0$,
  on   \'etudie  syst\'ematiquement la dimension des sous-espaces lin\'eaires  $\P^r_{k}\subset \P^n_{k}$ sur lesquels une forme
     $\lambda f + \mu g$ dans le pinceau de quadriques
     contenant $X$ peut s'annuler.  On \'etudie l'existence  de points
     quadratiques, de droites, de coniques sur $X \subset \P^n_{k}$.

Au \S \ref{para4}, on s'int\'eresse \`a $X \subset \P^4$. On montre   que,
sur un corps de nombres, si $X$ poss\`ede des points dans tous les compl\'et\'es,
alors l'indice de la surface $X$ divise 2. C'est un cas particulier d'un r\'esultat
de \cite{CV21}.

     Au \S \ref{para5}, on discute le cas $X \subset \P^5_{k}$, sous plusieurs angles.
     On \'etablit le th\'eor\`eme  \ref{quadglobal}. La d\'emonstration repose sur le fait
     que pour toute $X$, la vari\'et\'e $F_{1}(X)$ des droites contenues dans $X$ est connue, et qu'en
     particulier c'est une vari\'et\'e g\'eom\'e\-tri\-quement int\`egre, donc qui sur un corps de nombres
     poss\`ede  des points dans presque tous les compl\'et\'es de $k$.
     En utilisant le th\'eor\`eme \ref{sal93} (Salberger),
      on donne une d\'emonstration
     du th\'eor\`eme  \ref{PariIyer1} de Iyer et Parimala.  La seconde partie du \S \ref{para5} se place du point de vue
  g\'eom\'etrique  de la vari\'et\'e $G_{2}(X)$ des 2-plans contenus dans une quadrique du pinceau.
 On apporte des compl\'e\-ments \`a un article
   de Hassett et Tschinkel \cite{HT21}, sur un corps quelconque, par exemple le th\'eor\`eme
   \ref{alphabis}.
   Sur un corps de nombres, cela donne une autre d\'emonstration 
   du th\'eor\`eme  \ref{PariIyer1} de Iyer et Parimala, reposant elle aussi sur le
   th\'eor\`eme  \ref{sal93}.

     Dans le bref \S \ref{para6}, on consid\`ere le cas $X \subset \P^6_{k}$,
     sur lequel on n'a aucun r\'esultat global significatif autre que ceux
     obtenus dans \cite{CTSaSD87}.

   Au \S  \ref{para7}, on donne  {\it une d\'emonstration  relativement courte des  th\'eor\`emes \ref{hbrloc} et \ref{hbrgl} 
    de
  Heath-Brown \cite[Thm. 1, Thm. 2]{HB18} pour $X$ lisse dans $\P^7_{k}$.}
   Notre  d\'emonstration du th\'eor\`eme local \ref{hbrloc}  
  repose  sur un argument g\'eom\'etrique simple
  portant sur la g\'eom\'etrie de l'intersection de $X \subset \P^7_{k}$
   avec son espace tangent en un point
  rationnel,  et sur  notre  g\'en\'eralisation du  th\'eor\`eme de Creutz et Viray
  sur les points quadratiques sur un corps local dans le cas  d'une intersection
  quelconque de deux quadriques dans $\P^4_{k}$, laquelle repose ultimement
  sur la proposition \ref{intP3intro}.
Notre passage du th\'eor\`eme local  \ref{hbrloc} au  th\'eor\`eme local-global 
  \ref{hbrgl} repose sur le  Th\'eor\`eme  \ref{sal93} (Salberger), qui n'\'etait pas
  utilis\'e dans \cite{HB18}.
   
  \medskip
  
  Dans  l'appendice,  Alexander Kuznetsov  \'etablit  des propri\'et\'es g\'eom\'etriques
 des sch\'emas de Hilbert param\'etrant les espaces lin\'eaires et les quadriques contenus
 dans une intersection compl\`ete lisse de deux quadriques, telles qu'affirm\'ees au \S \ref{para2}.

\subsection{Notation}
  
  \'Etant donn\'e un corps $k$, on note $k^s$ une cl\^{o}ture s\'eparable
  et $\k$ une cl\^{o}ture alg\'ebrique. On note $\g={\rm Gal}(k^s/k)$ le groupe
  de Galois absolu. Pour une $k$-vari\'et\'e alg\'ebrique $X$, on note $\cl{X}=X \times_{k}  \k$.

  Pour $X$ une vari\'et\'e alg\'ebrique sur un corps $k$ et $L/k$ une extension de corps,
  on note $X(L)$ l'ensemble des points de $X$  \`a valeurs dans $L$.
  
  Un corps $k$ est dit fertile si pour toute $k$-vari\'et\'e lisse int\`egre $X$ sur $k$,
de dimension au moins 1, l'hypoth\`ese $X(k) \neq \emptyset$ implique
$X(k)$ Zariski dense. Un corps $p$-adique est fertile. Le corps des r\'eels est fertile.
Soit $p$ un nombre premier. Tout corps dont le groupe de Galois absolu 
est un pro-$p$-groupe est fertile. Cette notion est particuli\`erement int\'eressante
pour les intersections de deux quadriques : voir la remarque suivant la proposition
\ref{springer}.
  
  Pour $k$ un corps de nombres, et $v$ une place de $k$, on note $k_{v}$ le compl\'et\'e.
  On note $\A_{k}$ l'anneau des ad\`eles de $k$.  Pour $X$ une $k$-vari\'et\'e, on note $X(\A_{k})$
  l'ensemble de ses points ad\'eliques et on note  $X(\A_{k})^{\rm Br} \subset X(\A_{k})$ l'ensemble de
  Brauer--Manin de $X$ \cite[Chapter 13.3]{CTSk21}.

  On suppose le lecteur familier avec la th\'eorie alg\'ebrique des formes quadratiques \cite{Lam73, Lam05, Ka08}.
  Soit $r \geq 1$ un entier.
  Une forme quadratique non d\'eg\'en\'er\'ee  sur un corps $k$ 
  (${\rm{car}}(k)\neq 2$)
  est dite contenir  $rH$ si elle contient en facteur direct orthogonal la somme directe
  orthogonale de $r$ plans hyperboliques $H=<1,-1>$.

 Soit $n \geq 2$.  Une quadrique $C \subset \P^n_{k}$ de dimension $r-1$ est par d\'efinition une sous-vari\'et\'e
d'un espace lin\'eaire $\P^r_{k} \subset \P^n_{k}$ d\'efinie par l'annulation
d'un polyn\^ome homog\`ene de degr\'e 2 dans $\P^r_{k}$.

   \section{Rappels et pr\'eliminaires}\label{para2}

\subsection{Quelques \'enonc\'es connus}
    
   \begin{prop} Soit $k$ un corps, ${\rm car}(k) \neq 2$.
   Soit  $q(x_{0}, \dots,x_{n})$, $n \geq 1$  une forme quadratique non d\'eg\'en\'er\'ee sur $k$.
   Soit $Q \subset \P^n_{k}$ la quadrique qu'elle d\'efinit. Pour  un entier $r\geq 0$,
   les conditions suivantes sont \'equivalentes :
  
  (i) La forme $q$ contient $(r+1)H$.

  (ii) Il existe un $\P^r_{k}$ contenu dans $Q$.
   \end{prop}

   La   proposition suivante remonte \`a F. Ch\^{a}telet (1948). 
 Voir aussi  \cite[Chap. XII, \S 2, Thm. 2.1]{Lam05},
  \cite[Thm. 2.5]{CTSk93}, \cite[Prop. 8.1.10]{Ka08}.

 \begin{prop}\label{chatelet} 
 Soit $k$ un corps, ${\rm car}(k) \neq 2$. Soit 
 $q$ une forme quadratique  non d\'eg\'en\'er\'ee en 4 variables.   Soit $d$ son d\'eterminant.
 La forme quadratique $q$ a un z\'ero non trivial si et seulement si
 elle a un z\'ero non trivial sur le corps $k(\sqrt{d})$.
  \end{prop}
  \begin{proof}
  Il suffit de consid\'erer le cas o\`u $d$ n'est pas un carr\'e.
  Soit $K=k(\sqrt{d})$.
  On peut supposer la forme donn\'ee sous la forme $q=<1,-a,-b,abd>$ et que
  la forme $<1,-a,-b,ab>$ sur $K$ a un z\'ero non trivial, donc est hyperbolique.
  La quadrique $Q \subset \P^3_{k}$ d\'efinie par $q=0$  v\'erifie alors
  $Q_{K} \simeq \P^1_{K} \times_{K} \P^1_{K}$, les deux syst\`emes de g\'en\'eratrices
  sur $K$ \'etant permut\'ees par l'action de $Gal(K/k)$.
  Si $L \subset Q_{K}$ est une g\'en\'eratrice
  d'un certain type, la conjugu\'ee $L'$ sous l'action de $Gal(K/k)$ est
  une g\'en\'eratrice de l'autre type. Alors l'intersection de $L$ et $L'$
  est un $K$-point invariant sous $Gal(K/k)$, donc $Q$ poss\`ede un $k$-point.
  \end{proof}

   Le cas $r=0$ du th\'eor\`eme suivant fut \'etabli  ind\'ependamment par A. Brumer.

   \begin{prop}(Amer)\label{amerbrumer} \cite{Leep}
   Soit $k$ un corps, ${\rm car}(k) \neq 2$.
   Soit $X \subset \P^n_{k}$ une intersection  quelconque de deux quadriques d\'efinie par
   l'annulation de deux formes quadratiques $f$ et $g$. Soit $r\geq 0$ un entier.
    Les conditions suivantes
   sont \'equivalentes :
   
   (a) Il existe un espace lin\'eaire $\P^r_{k}$ contenu dans $X$.
   
   (b) La forme quadratique g\'en\'erique $f+tg$ sur le corps $k(t)$ s'annule sur
   un espace lin\'eaire $\P^r_{k(t)}$.
   
   Si la forme  quadratique g\'en\'erique $f+tg$ sur le corps $k(t)$
  est non d\'eg\'en\'er\'ee, par exemple si $X$ est une intersection compl\`ete lisse,
  ces conditions sont encore \'equivalentes \`a:
  
  (c) La forme quadratique  $f+tg$ contient $(r+1)H$.
    \end{prop}
    
En combinant  avec le th\'eor\`eme bien connu de T. A. Springer sur  les formes quadratiques,
    ceci donne :
    \begin{prop}\label{springer}
     Soit $k$ un corps, ${\rm car}(k) \neq 2$.
     Soit $K/k$ une extension finie de corps de degr\'e impair. Soit $X \subset \P^n_{k}$ une intersection  
     de deux quadriques d\'efinie par
   l'annulation de deux formes quadratiques $f$ et $g$. 
   Si les conditions (a) ou (b)
   de la proposition \ref{amerbrumer} valent apr\`es extension de $k$ \`a $K$,
   alors elles valent sur $k$.
 \end{prop}

 Une cons\'equence de ce r\'esultat est que, pour \'etablir une telle propri\'et\'e
 pour un $X$ sur $k$ donn\'e, on peut supposer le corps $k$ fertile. Il suffit en effet
 d'\'etablir le r\'esultat sur la perfection $E$ de $k$, puis sur le corps fixe
 d'un pro-2-Sylow du  groupe de Galois absolu de $E$.

 \begin{prop}\label{corpsCi}
 Soit $k$ un corps, ${\rm car}(k) \neq 2$. Soit $X \subset \P^n_{k}$
 une intersection quelconque de deux quadriques, d\'efinie par un syst\`eme $f=g=0$ de deux
 formes quadratiques.  
  
  (a) Si $k$ est alg\'ebriquement clos, et $n \geq 2r+2$, il existe un  sous-espace lin\'eaire
 $\P^r_{k}$  de $\P^n_{k}$ contenu dans $X$.
 
 (b) Si $k$ est fini, et $n \geq 2r+4$,il existe un  sous-espace lin\'eaire
 $\P^r_{k}$  de $\P^n_{k}$ contenu dans $X$.
\end{prop}
\begin{proof} Ceci r\'esulte de la proposition 
\ref{amerbrumer} et
du fait que   le corps $k(t)$ est un corps $C_{1}$ dans le cas (a)
et un corps $C_{2}$  dans le cas (b).
\end{proof}

 Pour une intersection quelconque de deux quadriques $X \subset \P^n_{k}$  avec $n \geq 2r+2$,
 la vari\'et\'e $F_{r}(X)$ est non vide. Soit en effet $f=g=0$ un syst\`eme de formes quadratiques d\'efinissant 
 $X \subset \P^n$. D'apr\`es la proposition \ref{amerbrumer}, pour voir que $F_{r}(X)(k)$ est non vide pour $k$ alg\'ebriquement clos,
  il suffit de voir que la forme quadratique $f+tg$  s'annule sur un espace lin\'eaire $\P^r_{k(t)} \subset \P^n_{k(t)}$. Ceci r\'esulte du fait que le corps  $k(t)$ est un corps $C_{1}$.

 \begin{prop}\cite[Prop. 2.1]{Reid72}
 Soit $k$ un corps, ${\rm car}(k) \neq 2$.
 Soient $X \subset \P^n_{k}$, $n \geq 2$, une intersection compl\`ete de deux quadriques
 d\'efinie par l'annulation de deux formes quadratiques $f(x_{0},\dots, x_{n})$ et $g(x_{0},\dots, x_{n})$.
 La $k$-vari\'et\'e est lisse si et seulement si le polyn\^{o}me homog\`ene  $det(\lambda f+ \mu g)$
est non nul et est s\'eparable, i.e. a  toutes ses racines simples dans $\k$.
  \end{prop}
  
  \medskip

Soit $X \subset \P^n_{k}, n \geq 4,$ 
une intersection compl\`ete lisse de deux quadriques.
Soit $P \in X(k)$. On peut supposer $X$ donn\'ee par un syst\`eme 
$ f= x_{0}x_{1} + q_{1}(x_{1}, \dots, x_{n})=0$
$ g=x_{0}x_{2} + q_{2}(x_{1}, \dots, x_{n})=0,$
Le point $P$ \'etant donn\'e par $(1,0,\dots,0)$ et l'espace tangent $T_{P}$ \`a $X$
en $P$ \'etant donn\'e par $x_{1}=x_{2}=0$.
Soit $H \subset \P^n_{k}$ l'espace projectif de dimension $n-3$
donn\'e par  $x_{0}=x_{1}=x_{2}=0$.  Soit 
$Y \subset H \simeq \P^{n-3}$  l'intersection de deux quadriques donn\'ee 
  par le syst\`eme
$$ q_{1}(0,0, x_{3}, \dots, x_{n})= q_{2}(0,0, x_{3}, \dots, x_{n})=0.$$

\begin{prop}\label{para3CTSaSD}
Soit $k$ un corps, ${\rm car}(k) =0$.
Avec les notations ci-dessus, l'intersection  $X \cap T_{P} \subset \P^n_{k}$
est le c\^{o}ne de sommet $P$ sur $Y$. 
  Les  droites de $X \subset \P^n_{k}$ qui passent par $P$ sont les g\'en\'eratrices de ce c\^one.
\end{prop}
\begin{proof} Voir \cite[\S  3,  Proof of Thm. 3.2]{CTSaSD87}.
\end{proof}

\subsection{Quelques lemmes pr\'eparatoires}

\begin{prop}\label{pointrat2H}
Soit $k$ un corps, ${\rm car}(k) =0$.   
Soit $X \subset \P^n_{k}$, $n \geq 3$, une intersection compl\`ete lisse
de deux quadriques, d\'efinie par un syt\`eme $f=g=0$.
Supposons $X(k)\neq \emptyset$. Pour $n=3$, supposons $k$ fertile.
Alors  il existe une forme non d\'eg\'en\'er\'ee $\lambda f + \mu g=0 $ dans le pinceau
qui s'annule sur un  $\P^1_{k}$,  c'est-\`a-dire qui contient $2H$,
et l'ensemble des 
 $(\lambda,\mu) \in \P^1(k)$ avec cette
propri\'et\'e est  Zariski dense dans $\P^1_{k}$.
\end{prop}

\begin{proof}
 Il existe un ouvert de Zariski non vide $U \subset X \times_{k}X$
tel que, pour tout couple de points g\'eom\'etriques $(A,B) \in U$, on ait $A \neq B$ et
la droite $AB$ n'est pas contenue dans $X$.
Ce dernier point est clair, car pour tout point g\'eom\'etrique $A \in X$ il existe un point g\'eom\'etrique $B$ de $X$
tel que la droite $AB$ ne soit pas contenue dans $X$, sinon  $X$ serait  un c\^{o}ne de sommet $A$. 
 On dispose alors d'un $k$-morphisme $\phi: U \to \P^1_{k}$ envoyant $(A,B) \in U$ sur le param\`etre $(\lambda,\mu)$ de  l'unique quadrique du pinceau contenant la droite $AB$.
 Pour $n \geq 4$ et $X(k)\neq \emptyset$, l'ensemble $X(k)$ est Zariski dense dans $X$, car $X$
 est $k$-unirationnelle \cite[Prop. 2.3]{CTSaSD87}. Pour $n=3$ et $k$ fertile, $X(k)\neq \emptyset$
 implique $X(k)$ Zariski dense dans $X$.
 On voit alors que $U(k)$ est Zariski dense dans $U$, et donc son image par  $\phi$ dans $\P^1(k)$
 est Zariski dense. 
 Comme l'ensemble des points $(\lambda,\mu) \in \P^1$ tels que  la forme quadratique $\lambda f + \mu g$
est singuli\`ere est fini, ceci \'etablit la proposition.
\end{proof}

\begin{prop}\label{pointquaddroite}
Soit $k$ un corps, ${\rm car}(k) =0$.   Soit $X \subset \P^n_{k}$,   
une intersection compl\`ete lisse de deux quadriques d\'efinie par $f=g=0$.
Supposons $n \geq 4$, ou $n \geq 3$ et le corps $k$   fertile, par exemple   $p$-adique.
Il  y a \'equivalence entre :

(a) Il existe un point de degr\'e au plus 2 sur $X$.

(b) Il existe une forme  $\lambda f + \mu g=0 $ dans le pinceau qui s'annule sur un
  $\P^1_{k}$.

(c) Il existe une forme non d\'eg\'en\'er\'ee $\lambda f + \mu g=0 $ dans le pinceau
qui s'annule sur un  $\P^1_{k}$,  c'est-\`a-dire qui contient $2H$.
  
L'ensemble des  $(\lambda,\mu) \in \P^1(k)$ satisfaisant cette
propri\'et\'e est 
vide ou
Zariski dense dans $\P^1_{k}$.
\end{prop}

\begin{proof} Que (c) implique (b) qui implique (a) est clair.
 Montrons que (a)
implique (c).
Si $X(k) \neq \emptyset$, il suffit d'appliquer la proposition \ref{pointrat2H}.
Soit $K/k$ une extension quadratique de corps, avec $X(K) \neq \emptyset$.
Dans la d\'emonstration de la proposition \ref{pointrat2H}, on peut remplacer $X \times_{k}X$
par $R_{K/k}(X_{K})$. On obtient un ouvert $U \subset R_{K/k}(X_{K})$
et un $k$-morphisme $\phi : U \to \P^1_{k}$.
La $k$-vari\'et\'e lisse $R_{K/k}(X_{K})$ contient un $k$-point.
Pour $n \geq 3$ et $k$ fertile, cela conclut la d\'emonstration.
Pour $n\geq 4$, l'intersection lisse de deux quadriques $X_{K}$ contient un
$K$-point, donc est $K$-unirationnelle. Ceci implique que la $k$-vari\'et\'e
$R_{K/k}(X_{K})$  
 est $k$-unirationnelle, et on conclut comme dans la d\'emonstration de la proposition \ref{pointrat2H}.
\end{proof}

  \medskip

\begin{lemma}\label{intersecpasintegre}
Soit $k$ un corps, ${\rm car}(k) =0$.
Soit $n \geq 3$ un entier. Soient $f$ et $g$ deux formes quadratiques sur $k$  en $n+1$ variables, et soit $X \subset \P^n_{k}$ la vari\'et\'e d\'efinie par $f=g=0$.
Si $X$ ne poss\`ede pas de point dans une extension de degr\'e 1 ou 2 de $k$, alors :

(i) La $k$-vari\'et\'e $X$ est non conique et purement de codimension 2 dans $\P^n_{k}$.

(ii)  Le polyn\^ome
$det(\lambda f + \mu g)$ est non nul.

(iii) Toute forme $\lambda f + \mu g$
dans le pinceau sur $k$ est de rang au moins 3.

(iv) Si   $n \geq 4$,  la $k$-vari\'et\'e $X$ est g\'eom\'etriquement int\`egre.

(v) Si $n=3$ et $X$ n'est pas une courbe lisse g\'eom\'etriquement int\`egre, alors $\cl{X}$ est la r\'eunion de 4 droites transitivement permut\'ees
par l'action du groupe de Galois, l'ensemble des 4 droites formant un cycle. De plus dans ce cas il n'existe
pas de forme de rang 3 dans le pinceau sur $\k$.
\end{lemma} 
\begin{proof}
Si $X$  est conique, le sommet du c\^{o}ne est un
espace lin\'eaire d\'efini sur $k$.

Si $f$ et $g$  ont un facteur commun, on peut le supposer d\'efini sur $k$, il  est de degr\'e 1 ou 2, 
donc $X$ poss\`ede un point dans une extension $K/k$ avec $[K:k] \leq 2$.

Supposons d\'esormais $f$ et $g$ sans facteur commun.
Alors  $X \subset \P^n_{k}$ est purement de codimension 2 \cite[Lemma 1.1]{CTSaSD87}.
 
Si le polyn\^{o}me homog\`ene $det(\lambda f + \mu g)$ s'annule identiquement, alors
$X$ poss\`ede un point singulier rationnel \cite[Lemma 1.14]{CTSaSD87}.

S'il existe une forme de rang au plus 2 dans le pinceau $\lambda f + \mu g$,
elle s'annule sur un espace lin\'eaire $\P^{n-2}_{k}  \subset \P^n_{k}$, et $X$
contient la trace de toute autre forme du pinceau, donc une quadrique
de dimension $n-3\geq 0$, donc un point dans une extension au plus quadratique de $k$.

On est   ramen\'e \`a \'etudier le cas o\`u $f$ et $g$ n'ont pas de facteur commun,
o\`u il existe une forme de rang 
 $n+1$ dans le pinceau, et  o\`u toute forme non
nulle dans le pinceau est de rang au moins 3. 

Pour $n \geq 4$,  
  \cite[Lemma 1.11]{CTSaSD87} assure que la $k$-vari\'et\'e $X$ est g\'eom\'e\-tri\-quement int\`egre.
  
  Soit $n=3$. Si la courbe $X$ est singuli\`ere mais g\'eom\'etriquement int\`egre, comme c'est une courbe
     de genre arithm\'etique 1, elle poss\`ede exactement un point singulier, qui est donc rationnel.
 Supposons $\cl{X}$   non int\`egre. Si $X$ contient une droite d\'efinie sur $k$
 ou une extension quadratique de $k$, ou si $X$ contient une conique lisse, l'\'enonc\'e est clair.
  En combinant \cite[Lemmas 1.7, 1.10]{CTSaSD87} et le fait que toute forme dans le pinceau
  $\lambda f + \mu g$ sur $k$ est de rang au moins 3, on voit qu'on est r\'eduit
  au cas o\`u  $\cl{X}$ est  l'union de 4 droites distinctes transitivement permut\'ees par
  l'action du groupe de Galois.
  Soit $Q \subset \P^3_{k}$  une quadrique lisse d\'efinie par une forme du pinceau.
  On a $\Pic(\cl{Q}) = \Z e_{1} \oplus \Z e_{2}$, les classes $e_{1}$ et $e_{2}$ sont permut\'ees par
  l'action du groupe de Galois, deux droites sont dans la classe $e_{1}$, deux dans la classe $e_{2}$.
  Chaque droite rencontre exactement deux autres des 4 droites.
  S'il existe une forme de rang exactement 3 dans le pinceau sur $\k$,
  alors les 4 droites sont contenues dans un c\^{o}ne sur une conique lisse.
  Toute droite de $\P^3$ contenue dans ce c\^{o}ne est une g\'en\'eratrice,
   i.e. passe par le sommet du c\^{o}ne.
  Mais alors   les 4 droites sont concourantes, contradiction.
\end{proof}

\begin{prop}\label{coraytsfasman}
Soit $k$ un corps, ${\rm car}(k) =0$.
Soit $X \subset \P^4_{k}$ une intersection 
compl\`ete de deux quadriques, g\'eom\'etriquement int\`egre et non conique. Si $X$ est 
singuli\`ere, alors
soit  il existe une extension $K/k$ de degr\'e $[K:k]\leq 2$
telle que $X_{K}$ est $K$-birationnelle \`a $\P^2_{K}$, soit 
$X$ est $k$-birationnelle \`a une surface
$R_{K/k}(C)$ pour $K/k$ une extension quadratique de corps et $C$ une conique lisse sur $K$.
\end{prop}
\begin{proof}
Si $C \subset X$ est une courbe (non n\'ecessairement g\'eom\'etriquement int\`egre)
contenue dans le lieu singulier de $X$, alors $C$ est une droite  $\P^1_{k} \subset \P^4_{k}$
\cite[Lemma 3.15.1]{CTSaSD87}. Dans ce cas, $X$ est $k$-birationnelle au produit
d'un espace projectif  et d'une quadrique lisse de dimension au moins 1
\cite[Prop. 2.2]{CTSaSD87}.
On peut donc supposer que les singularit\'es de $X$
sont isol\'ees.  
Ces singularit\'es
sont discut\'ees par Coray et Tsfasman, en particulier
du point de vue de l'action du groupe de Galois. 
Notons qu'une surface dite d'Iskovkikh devient $K$-rationnelle sur l'extension
quadratique $K/k$ sur laquelle sont d\'efinis ses points doubles.
Il ressort de \cite[\S 6, 7]{CoTs88} en particulier de \cite[Lemma 7.4]{CoTs88} que,
sauf peut-\^{e}tre dans le cas o\`u $X$ poss\`ede 
4 points singuliers de type $4 A_{1}$, il existe une extension $K/k$ au plus
quadratique telle que $X_{K}$ soit $k$-birationnelle \`a $\P^2_{K}$.
Dans le cas 
$4A_{1}$, la surface est $k$-birationnelle \`a une surface de del Pezzo de degr\'e 8,
c'est-\`a-dire \`a une vari\'et\'e $R_{K/k}(C)$ pour $K/k$ extension quadratique s\'eparable et $C$ conique lisse sur $K$.
\end{proof}

 \begin{rmk}
 Pour $k$ $p$-adique, on peut \'eliminer le cas o\`u $X$ est  $k$-birationnelle \`a $Y=R_{K/k}(C)$,
 avec $C/K$ conique lisse sur un corps extension quadratique de $k$.
  Comme $k$ est un corps $p$-adique, il existe une extension quadratique $L/k$ de corps
  telle que $K \otimes L / K $ soit une extension quadratique de corps $KL/K$.
On a $R_{K/k}(C) (L) = C(K \otimes _{k} L) = C(KL)$.  Mais toute conique sur
un corps $p$-adique acquiert un point rationnel dans toute extension  de corps quadratique.
Donc  $C(KL) \neq \emptyset$ et $R_{K/k}(C)$ poss\`ede un point dans $L$, 
donc est $L$-birationnel \`a $\P^2_{L}$,
et les points $L$-rationnels  sont denses sur $X_{L}$.
\end{rmk}

  \begin{rmk}
 Il y a une certaine analogie entre la  proposition \ref{coraytsfasman} et la  premi\`ere d\'emonstration de  \cite[Thm. 2.1]{CV21},
 qui repose sur une discussion des d\'eg\'en\'erescences possibles d'une surface de del Pezzo de degr\'e 4.
 \end{rmk}

\subsection{G\'eom\'etrie des vari\'et\'es d'espaces lin\'eaires}

\begin{prop}\label{generalintegre}  
Soit $k$ un corps, ${\rm car}(k) =0$.
Soit  $X \subset \P^n_{k}$ une intersection compl\`ete
lisse de deux quadriques. Soit 
$F_{r}(X)$  la vari\'et\'e des espaces lin\'eaires de dimension $r$
contenus dans $X$. 

(a) Si l'on a $n\leq 2r+1$, alors $F_{r}(X)$ est vide.
Si l'on a $n \geq 2r+2$, alors $F_{r}(X)$ est non vide,
 projectif et  lisse, de dimension $(r+1)(n-2r-2)$,
 g\'eom\'etriquement connexe si $n>2r+2$.

(b) Pour $n=5$, la vari\'et\'e $F_{1}(X)$ est un espace principal homog\`ene
d'une vari\'et\'e ab\'elienne de dimension 2.

(c) Pour $n=6$, la vari\'et\'e  $F_{1}(X)$ est 
une vari\'et\'e de Fano projective lisse, g\'eom\'etriquement rationnelle, de dimension 4,
dont le groupe de Picard g\'eom\'etrique est libre de rang~8. 

(d) Pour $n=7$, la vari\'et\'e $F_{1}(X)$ est une vari\'et\'e de Fano
projective lisse, g\'eom\'etriquement rationnelle, de dimension 6,
 dont le groupe de Picard g\'eom\'etrique est libre de rang~1.

(e)  Pour $n=7$, la vari\'et\'e $F_{2}(X)$   est un espace principal homog\`ene  
d'une vari\'et\'e ab\'elienne de dimension 3.

(f)  Soit $n \geq 5$. Soit $Z \subset X \times_{k} F_{1}(X)$ la vari\'et\'e d'incidence des couples 
$(x,L)$ form\'ee d'un point de $X$ et d'une droite $L \subset X$ avec $x\in L$.
La projection $Z \to F_{1}(X)$ est un fibr\'e projectif en $\P^1$ localement trivial.
La $k$-vari\'et\'e $Z$ est lisse et  g\'eom\'etriquement connexe, de dimension $2n-7$.
Pour $n \geq 5$,
les fibres g\'en\'erales g\'eom\'etriques de la projection $Z \to X$ sont des intersections compl\`etes
lisses de deux quadriques dans $\P^{n-3}$. Pour $n \geq 6$,
ces fibres sont g\'eom\'etriquement connexes.
\end{prop}
\begin{proof}
 Pour \'etablir ces \'enonc\'es, on peut supposer le corps  $k$  alg\'ebri\-quement clos.

Voici des r\'ef\'erences pour les  divers r\'esultats ci-dessus, et des pr\'ecisions sur
les cas o\`u les  r\'esultats sont disponibles en caract\'eristique positive.

L'\'enonc\'e (a), sous la seule hypoth\`ese ${\rm car}(k)\neq 2$,
est \'etabli par A. Kuznetsov dans l'appendice.
Voir les Lemmes  \ref{lem:fr} et \ref{lem:fr-irred}. 
  Le fait que $F_{r}(X)$ est non vide
pour   $n \geq 2r+2$ vaut pour toute intersection de deux quadriques
(Proposition \ref{corpsCi}).

 Xiaoheng Wang \cite{XW18} \'etudie  les propri\'et\'es des vari\'et\'es d'espaces lin\'eaires de dimension maximale contenus dans une intersection compl\`ete  lisse  $X$ de deux quadriques dans $\P^n_{k}$, $n \geq 3$
et tout corps  $k$ avec ${\rm car.}(k) \neq 2$. Pour $X$ de dimension impaire sur un tel corps,
il montre que cette vari\'et\'e est un espace principal homog\`ene d'une vari\'et\'e ab\'elienne.  
Ceci vaut en particulier pour $F_{1}(X)$ et $n=5$  et pour $F_{2}(X)$ et $n=7$,
c'est-\`a-dire les \'enonc\'es (b) et  (e).
 
 En caract\'eristique z\'ero,  Araujo-Casagrande  \cite{AC17}  \'etudient la vari\'et\'e  des espaces lin\'eaires de dimension $m-1$
contenus dans une intersection compl\`ete lisse de deux quadriques  $X \subset \P^{2m+2}$
pour $m \geq 2$. C'est une vari\'et\'e projective, lisse, g\'eom\'etriquement connexe  
 de dimension $2m$. C'est une vari\'et\'e de Fano qui est g\'eom\'etriquement rationnelle.
 Pour $m=1$, il s'agit des surfaces de del Pezzo de degr\'e 4.
Le cas qui nous int\'eresse ici est $m=2$, c'est-\`a-dire la vari\'et\'e $ F_{1}(X)$ pour $X \subset \P^6$.
Le groupe de Picard g\'eom\'etrique dans ce cas  est $\Z^8$. Ceci donne l'\'enonc\'e (c).
 
 En caract\'eristique z\'ero,  la vari\'et\'e $F_{1}(X)$ des $\P^1$ dans $X \subset \P^7$ 
est une vari\'et\'e projective, lisse, g\'eom\'etri\-quement connexe,
de Fano \cite[Rem. 3.2, Rem. 3.6.1]{DM98},
 g\'eom\'etri\-quement rationnelle  \cite{N75} \cite[Rem. 7.3]{DM98},
 et le groupe de Picard g\'eom\'e\-trique est $\Z$ \cite[Cor. 3.5]{DM98}.
 Ceci donne l'\'enonc\'e (d).

\'Etablissons  le point (f). 
L'espace total $Z$ est lisse de dimension $2n-7$.
Pour $n\geq 5$, par tout point de $X$ il passe une droite,
le morphisme propre
  $q: Z \to X$ est donc dominant, la fibre g\'en\'erique de $q$ est
  de dimension $n-5$.
  Comme le corps  $k$ est de caract\'eristique z\'ero, la fibre g\'en\'erique de $q$  est 
  lisse. Pour tout point de $X$, d'apr\`es  la proposition \ref{para3CTSaSD},
  la fibre est une intersection de deux quadriques dans $\P^{n-3} $.
  Ainsi la fibre g\'en\'erique de  $Z \to X$ est une intersection 
  compl\`ete lisse de deux quadriques dans $\P^{n-3} $, elle est donc g\'eom\'etriquement
  int\`egre si l'on a $n \geq 6$. Il en est donc de m\^eme des fibres en tout point
  sch\'ematique d'un ouvert de Zariski non vide de $X$. 
 \end{proof}
    
 \begin{rmk}  Sur un corps  $k$ avec ${\rm car}(k)=0$,
Debarre-Manivel \cite{DM98} \'etudient les vari\'et\'es des sous-espaces lin\'eaires des intersections compl\`etes
{\it g\'en\'erales}. Certains de leurs r\'esultats  valent en caract\'eristique positive.
\end{rmk}

  Soit $X \subset \P^n_{k}$ une intersection compl\`ete  lisse 
  de deux quadriques donn\'ee par un syst\`eme  $f=g=0$. 
  Soit $G_{r}(X) \subset {\rm{Gr}}(r, n) \times \P^1_{k}$ la sous-vari\'et\'e 
  form\'ee des couples $(L, m)$ avec $L\subset \P^n$ espace lin\'eaire de dimension $r$ contenu  dans la quadrique 
 $\lambda f + \mu g=0$, o\`u $m=(\lambda,\mu) \in \P^1_{k}$.
 Soit $S_{r}(X)$ le sch\'ema de Hilbert  des quadriques  de dimension $r-1$ contenues dans $X$.

 \medskip
   
   La proposition suivante rassemble des r\'esultats \'etablis par A. Kuznetsov
   dans l'appendice.
    
  \begin{prop}\label{kuznestov}
  Soit $k$ un corps, ${\rm car}(k) \neq 2$.
   Supposons $2r+1 \leq n$.
  Avec les notations ci-desssus :
  
(1)  La vari\'et\'e $G_{r}(X)$ est lisse et g\'eom\'etriquement connexe, 
 et le morphisme $G_{r}(X) \to \P^1_{k}$ dominant.

 (2)   La vari\'et\'e $S_{r}(X)$ des quadriques de dimension $r-1$
  contenues dans $X$ est lisse et 
   g\'eom\'etriquement connexe. 
   La sous-vari\'et\'e des quadriques
non lisses est un ferm\'e propre de $S_{r}(X)$.

  (3)      \`A un espace lin\'eaire $L\simeq \P^r_{k}$ contenu dans une quadrique du pinceau et non contenu dans $X$, on associe 
  $X \cap L$. Inversement, \'etant donn\'ee une quadrique (g\'en\'eralis\'ee)  de dimension   $r-1$ contenue dans $X$,
 engendrant un espace lin\'eaire $L$  de dimension $r$ non contenu dans $X$, on lui associe le couple form\'e 
de l'espace lin\'eaire $L$
 et du param\`etre $(\lambda,\mu)$ de l'unique  quadrique du pinceau $\lambda f + \mu g$ contenant~$L$.
 Ceci d\'efinit des applications rationnelles birationnelles inverses l'une de l'autre entre $G_{r}(X)$ et $S_{r}(X)$.
 
 (4)  Si l'on a $n=2r+1$, et donc $F_{r}(X)=\emptyset$, 
  alors les   $k$-vari\'et\'es  $G_{r}(X)$ et $S_{r}(X)$ sont isomorphes.
  \end{prop}
  \begin{proof}
 Pour l'\'enonc\'e (1), voir la  proposition \ref{prop:gr}.
Pour les \'enonc\'es (2), (3), (4)  voir  la proposition \ref{prop:sr}.
   \end{proof}

 \begin{prop}\label{coniqueplan}
 Soit $k$ un corps, ${\rm car}(k) =0$ et $n \geq 4$.  Soit $X \subset \P^n_{k}$,   
une intersection compl\`ete lisse de deux quadriques d\'efinie par $f=g=0$.
Il  y a \'equivalence entre :

(a) Il existe une conique $C \subset X$.

(b) Il existe une forme  $\lambda f + \mu g=0 $ dans le pinceau qui s'annule sur un
  plan $L \simeq\P^2_{k}$.
  
  Supposons de plus $n\geq 5$.
  Si   le corps $k$ est fertile, alors ces propri\'et\'es sont \'equivalentes \`a
  chacune des propri\'et\'es :
  
  (c) Il existe une forme  $\lambda f + \mu g=0 $ non d\'eg\'en\'er\'ee  dans le pinceau qui s'annule sur un
  plan $L \simeq\P^2_{k}$, c'est-\`a-dire qui contient $3H$.
  
  (d) Il existe une conique lisse $C \subset X$.
  
  (e) Il existe une forme $\lambda f + \mu g=0 $ non d\'eg\'en\'er\'ee  
  dans le pinceau qui s'annule sur un
  plan $L \simeq\P^2_{k}$ non contenu dans $X$,
  et telle que $X \cap L$ soit une conique lisse.
\end{prop}
\begin{proof} Sous l'hypoth\`ese (a), soit $L \subset \P^n_{k}$
 le plan de la conique, et $P \in L(k)$ un point non situ\'e sur la conique.
 Il existe une forme $\lambda f + \mu g$ dans le pinceau qui s'annule
 en $P$. Elle s'annule alors sur le plan $L$.
 Sous l'hypoth\`ese (b), soit on a $L \subset X$,
 et alors on a clairement des coniques lisses contenues dans $X$,
  soit toute quadrique du pinceau diff\'erente de
  $\lambda f + \mu g=0$ d\'ecoupe sur le  plan $L$  
 une conique contenue dans $X$. 
 
 On utilise maintenant la proposition \ref{kuznestov} dans le cas $r=2$.

 Sous l'hypoth\`ese (b), il existe un $k$-point sur 
$G_{2}(X)$, qui est une $k$-vari\'et\'e lisse g\'eom\'etriquement int\`egre. 
 Si $k$ est fertile,
les $k$-points sont Zariski denses sur $G_{2}(X)$. Le $k$-morphisme
$G_{2}(X) \to \P^1_{k}$ est dominant, il existe donc $(\lambda,\mu) \in \P^1(k)$
tel que $\lambda f + \mu g$ soit de rang maximal et s'annule sur un plan $L$.
Ceci \'etablit (c). 

Sous l'hypoth\`ese (a), comme la vari\'et\'e $S_{2}(X)$ est lisse et g\'eom\'etriquement int\`egre,
et que le lieu des coniques singuli\`eres est un ferm\'e propre de $S_{2}(X)$, si $k$ est
fertile, alors il existe une conique lisse dans $X$.

Pour obtenir (e), on utilise le fait que $G_{2}(X)$ et $S_{2}(X)$ sont
g\'eom\'etriquement int\`egres et
$k$-birationnelles entre elles,  et que le lieu des coniques non lisses
 est un ferm\'e propre de $S_{2}(X)$.
\end{proof}

 \begin{rmk}
 Pour $k$ de caract\'eristique z\'ero quelconque, les hypoth\`eses (a) ou (b)
 impliquent l'existence d'une extension finie de corps $K/k$ de degr\'e impair
 sur laquelle les \'enonc\'es (c), (d), (e)  valent. Ceci r\'esulte du cas $p=2$
 de l'\'enonc\'e g\'en\'eral  rappel\'e dans l'introduction :
 le corps fixe d'un pro-$p$-Sylow 
 du groupe de Galois absolu d'un corps  parfait $k$   est  un corps fertile.
 \end{rmk}

  \subsection{Principe local-global :  quelques r\'esultats connus}

  \medskip

  \begin{thm}\label{Hasse} (Hasse 1924 \cite{H24})
  Soient $k$ un corps de nombres, $1\leq n \leq m$ des entiers 
   et $\phi$ et $\psi$ deux formes quadratiques non d\'eg\'en\'er\'ees de rangs respectifs $n$ et $m$.
   Si sur tout compl\'et\'e $k_{v}$ de $k$ la forme $\phi$ est une sous-forme de $\psi$,
   alors sur $k$ c'est une sous-forme de $\psi$.
  \end{thm}
   Les cas classiques sont $n=1$ et $n=m$. Le cas g\'en\'eral
se d\'eduit du cas $n=1$ par  le th\'eor\`eme de simplification de Witt.

  \medskip
  
  L'\'enonc\'e suivant reformule un r\'esultat \'etabli ind\'ependamment
  dans   \cite{CT88} et \cite{Sal89}.
 
     \begin{thm}\label{sal93+ct}
Soit $k$ un corps de nombres. Soit 
 $X \subset \P^4_{k}$ une   intersection lisse de deux quadriques.
 Si $X$ contient une conique lisse $C \subset \P^4_{k}$,
et si $X(\A_{k})^{\Br} \neq \emptyset$, alors $X(k)$ est
non vide et dense dans $X(\A_{k})^{\Br} $.
\end{thm}

\begin{proof}
 Soit $K  \in \Pic(X)$ la classe du faisceau canonique.
C'est l'oppos\'e de la classe d'une section hyperplane.
On a $(K.K)=4$ et le genre arithm\'etique de $C$ est donn\'e par la formule $p_{a}(C)=(C.(C+K))/2 +1$, donc $(C.C)=0$.
Le th\'eor\`eme de Riemann-Roch donne $h^0(C)\geq 2$, et donc
$h^0(C)=2$. Le syst\`eme lin\'eaire associ\'e \`a $C$ d\'efinit alors
un morphisme $X \to \P^1_{k}$ dont les fibres g\'en\'erales sont des
coniques. Un calcul classique montre que sur la cl\^{o}ture alg\'ebrique
il y a exactement 4 fibres g\'eom\'etriques singuli\`eres, chacune form\'ee
d'un couple de droites se recontrant en un point.
Sur un corps de nombres, 
 il a \'et\'e \'etabli par Salberger \cite{Sal88} \cite{Sal89}
et aussi dans \cite{CT88}  que pour une telle surface $X$ fibr\'ee en coniques
avec au plus 4 fibres g\'eom\'etriques non lisses,  l'hypoth\`ese
$X(\A_{k})^{\Br} \neq \emptyset$ implique $X(k) \neq \emptyset$.
Plus pr\'ecis\'ement, $X(k)$ est dense dans $X(\A_{k})^{\Br}$.
\end{proof}

\bigskip
  
  Le th\'eor\`eme suivant a \'et\'e \'etabli par Salberger \cite{Sal93} en 1993.
  Dans son travail \cite{Ha94} sur   la m\'ethode des fibrations, 
  Harari en esquisse une d\'emonstration \cite[Prop. 5.2.6]{Ha94}.
  Pour la commodit\'e du lecteur, je donne ci-dessous les d\'etails
  de l'argument.

    \begin{thm}(Salberger 1993) \label{sal93}
Soit $k$ un corps de nombres. Pour tout entier $n \geq 5$ et
 toute intersection compl\`ete  lisse de deux quadriques $X \subset \P^n_{k}$
contenant une conique $C \subset \P^n_{k}$, le principe de Hasse vaut.
\end{thm}

\begin{proof}
 Soit $n \geq 5$. Si la conique est g\'eom\'etriquement r\'eductible, alors elle contient un point rationnel.
Supposons donc que $X$ contient une conique lisse $C$. Celle-ci est contenue dans
un   plan $\Pi:= \P^2_{k} \subset \P^n_{k}$ bien d\'efini. Si le plan $\Pi$ est contenu
dans $X$, alors $X(k)\neq \emptyset$. On suppose donc que $\Pi$ n'est pas contenu
dans $X$. L'intersection $\Pi \cap X$ est une courbe quartique dans $\Pi$, contenue dans  $X$,
et contenant la conique $C$.

Un th\'eor\`eme de Zak assure que  pour tout hyperplan $H$ de $\P^n$,
l'intersection compl\`ete  $X \cap H $ n'a qu'un nombre fini de singularit\'es et est
g\'eom\'etriquement int\`egre (cf. \cite[Prop. 5.2.1]{Ha94}).

Pour $n \geq 5$, une version du th\'eor\`eme de Bertini 
assure que l'hyperplan g\'en\'eral $H\simeq \P^{n-1}_{k}$ contenant $\Pi$, qui est param\'etr\'e 
par un espace $\P^{n-3}$, avec donc $n-3 \geq 2$,
d\'ecoupe sur $X $ une vari\'et\'e $H \cap X$  lisse et g\'eom\'etriquement connexe   (contenant la conique $C$).
 
On choisit deux tels  hyperplans, et on consid\`ere l'application rationnelle de $X$ vers $\P^1_{k}$
associ\'ee. C'est un morphisme hors de $X \cap \Pi$ (union dans $\Pi$ de $C$ et d'une conique).
Soit $Z \subset X \times \P^1_{k}$ le graphe de ce morphisme. 
D'apr\`es ce qui pr\'ec\`ede, les fibres  de $\pi: Z \to \P^1_{k}$ sont g\'eom\'etriquement int\`egres, 
donc $Z$ est g\'eom\'etriquement int\`egre.  
La projection $Z \to X$ est un morphisme $k$-birationnel.
 Au-dessus d'un ouvert  non vide $U \subset  \P^1_{k}$, la fibration $Z_{U}\to U$
est projective et lisse, \`a fibres des intersections compl\`etes  de deux quadriques dans $\P^{n-1}_{k}$,
g\'eom\'etriquement int\`egres, contenant $C$.
  
Soit $W$ une $k$-vari\'et\'e projective, lisse, g\'eom\'etriquement int\`egre
\'equip\'ee d'un  $k$-morphisme  $W \to Z$  birationnel 
qui induit un isomorphisme $W_{U} \to Z_{U}$ (possible par r\'esolution des singularit\'es).
Les fibres $W_{m}$ du morphisme compos\'e $W \to Z \to \P^1_{k}$ contiennent toutes une composante
de multiplicit\'e 1 g\'eom\'etriquement int\`egre sur le corps $k(m)$.
 
 Supposons $X(\A_{k}) \neq\emptyset$.
Pour $n \geq 5$, on a $\Br(X)/\Br(k)=0$.
 La m\'ethode des fibrations,   
 sous la forme donn\'ee par Harari
dans  \cite[Thm. 4.2.1]{Ha94}, et 
sous une forme encore plus g\'en\'erale
 dans  
  \cite[Thm. 1.3]{HWW21},
donne l'existence d'un point $m\in U(k)$ tel que la fibre $W_{m}$
satisfasse $\Br(W_{m})/\Br(k)=0$ et  $W_{m}(\A_{k}) \neq \emptyset$.

Par r\'ecurrence sur $n \geq 5$, le principe de Hasse pour $X$
se ram\`ene donc  au th\'eor\`eme \ref{sal93+ct}.
\end{proof}

\medskip

La m\'ethode montre aussi que $X(k)$ est dense dans $X(\A_{k})$,
mais sous l'hypoth\`ese $X(k)\neq \emptyset$, c'est un \'enonc\'e facile
\`a \'etablir directement \cite[Thm. 3.11]{CTSaSD87} pour toute   intersection compl\`ete lisse
de deux quadriques dans $\P^n_{k}$ avec $n \geq 5$.

 \begin{rmk}
 Dans son manuscrit \cite{Sal93}, pour $n \geq 6$
 et $X \subset \P^n_{k}$ une intersection compl\`ete
 de deux quadriques,
 g\'eom\'etriquement int\`egre et non conique,  mais non n\'ecessairement lisse,
sous l'hypoth\`ese que $X$ contient une conique lisse,
 Salberger \'etablit le principe de Hasse pour les
  mod\`eles projectifs et lisses de $X$.
   \end{rmk}

 Terminons cette section par deux \'enonc\'es connus sur les corps locaux, et qui seront
 utilis\'es dans l'article.
 Pour le premier, voir aussi
  \cite[Chap. VI, Cor. 2.5]{Lam73} et
 \cite[Chap. VI, Cor. 2.15]{Lam05}.

  \begin{prop}\label{bienconnu}
  Soit $k$ un corps local, ${\rm car}(k) \neq 2$. Soit 
 $q$ une forme quadratique non d\'eg\'en\'er\'ee en 4 variables.   
Si le d\'eterminant $d$ de $q$ n'est pas un carr\'e dans $k$,
alors  $q$ a un z\'ero non trivial sur $k$.
 \end{prop}
   \begin{proof}
 Soit $q=<1,-a,-b,abd>$ avec $d$ non carr\'e. Soit $K=k(\sqrt{d})$.
  D'apr\`es la proposition \ref{chatelet}, il suffit de montrer que la forme
   $q_{K}=<1,-a,-b,ab>$  est isotrope, c'est-\`a-dire hyperbolique. 
  C'est le cas si et seulement si l'alg\`ebre de quaternions $(a,b) \in \Br(k)$
  a une image triviale dans $\Br(K)$. Mais on sait bien que pour une
  extension   de corps $p$-adiques $K/k$ de degr\'e $n$, la restriction 
  $Br(k) \to \Br(K)$  s'identifie \`a la multiplication par $n$ sur $\Q/\Z$.
  \end{proof}

  \begin{prop}(Mordell)\label{mordell}
 Soit $X \subset \P^n_{\R}$ une intersection compl\`ete  lisse de  deux quadriques $f=g=0$.
 Dans le pinceau  de formes quadratiques $\lambda f+\mu g$ il existe une
 forme non singuli\`ere de signature $0$ ou $1$: il existe une forme non singuli\`ere dans le pinceau
 contenant  $[(n+1)/2] H$.
  \end{prop}
\begin{proof}
     On fait varier $(\lambda,\mu) $ dans le cercle $S^1$ d'\'equation 
     $\lambda^2+\mu^2=1$.
   Pour  $\lambda f+ \mu g$ non singuli\`ere, la signature de  $-\lambda f - \mu g$ 
   dans  $\Z$
   est l'oppos\'ee de la signature de $\lambda f+\mu g$. Par ailleurs au passage
   d'un point $(\lambda_{0}, \mu_{0}) \in \R$ o\`u la forme $\lambda_{0} f+\mu_{0} g$
   est singuli\`ere, la signature varie par addition ou soustraction de $2$.
     Voir \cite[\S10, Proof of Thm. 10.1, (e)  p. 114]{CTSaSD87} et \cite[Lemma 12.1]{HB18}. 
     \end{proof}

\section
{Points quadratiques sur les intersections de deux quadriques sur un corps $p$-adique}\label{para3}
 
 Soient $k$ un corps parfait, $\k$ une cl\^{o}ture alg\'ebrique,  et $\g=\Gal(\k/k)$.
Soit $X$ une courbe projective et lisse de genre 1 sur le corps $k$.
C'est un espace principal homog\`ene de sa jacobienne $J$.
On appelle p\'eriode de $X$ l'exposant de la classe $[X] \in H^1(k,J)$.
 On v\'erifie que cet entier est le g\'en\'erateur positif de l'image de
 l'application degr\'e 
 $\Pic(\cl{X})^\g \to \Z.$
 On appelle indice de $X$ le g\'en\'erateur positif de l'image de
 l'application degr\'e $\Pic(X) \to \Z$.  C'est  le pgcd des degr\'es
 des points ferm\'es sur $X$. Puisque $X$ est une courbe de genre 1,
c'est  aussi le plus petit degr\'e d'un tel 
 point ferm\'e. L'exposant divise l'indice. 
 L'\'enonc\'e suivant  
 est \cite[Thm. 3]{Li68} \cite[Thm. 7]{Li69}. La d\'emonstration utilise
 le th\'eor\`eme de dualit\'e de Tate pour les vari\'et\'es ab\'eliennes
 sur les corps locaux.
 \begin{thm}(Roquette, Lichtenbaum)\label{periode}
 Soit $X$ une courbe de genre 1 sur un corps $p$-adique $k$.
 L'exposant et l'indice de $X$ co\"{\i}ncident.
\end{thm}

Cet entier est aussi \'egal \`a l'ordre du noyau (fini) de la restriction $\Br(k)\to \Br(X)$
(Roquette, Lichtenbaum).

 La proposition suivante (\'enonc\'es (a) \`a (d)) est un cas particulier, sans doute classique,
d'un r\'ecent th\'eor\`eme de Xiaoheng Wang \cite{XW18} sur les sous-espaces lin\'eaires
de dimension maximale des intersections compl\`etes lisses de deux quadriques
dans $\P^{2n+1}$.  Une variante ant\'erieure est utilis\'ee 
par Creutz et Viray  dans la d\'emonstration de \cite[Prop. 4.7]{CV21}.

\begin{prop}\label{pseudoclass}
  Soit $k$ un corps de caract\'eristique diff\'erente de 2, et
 soit $X\subset \P^3_{k}$ une  intersection compl\`ete lisse
de deux quadriques $f=g=0$.  C'est une courbe de genre 1.
Notons $J=J_{X}=\Pic^0_{X/k}$. 
On a $X=\Pic^1_{X/k}$.
 Soit $C/k$ la courbe  projective lisse  d'\'equation
$y^2=det(\lambda f+ \mu g )$. C'est une courbe de genre 1.
 On a les propri\'et\'es suivantes :

(a) Les jacobiennes $J_{X}$ et $J_{C}$ sont isomorphes.

(b) On a $C \simeq \Pic^2_{X/k}$.

 (c) Dans $H^1(k,J_{X})  $, on a $[C]=2[X]$ et $2[C]=0$.
 
 (d) Si $C(k) \neq \emptyset$, alors la p\'eriode de $X$ divise 2.
 
 (e) Si $k$ est un corps $p$-adique, et  $C(k) \neq \emptyset$, alors 
  $X$ poss\`ede un point dans une
 extension quadratique de $k$.

 \end{prop}

\begin{proof} Pour les points (a) \`a (d), voir Wang 
\cite[ \S 2.2. Theorem 2.25. p. 372.]{XW18}.
 Le point (e) r\'esulte de (d) et du  th\'eor\`eme \ref{periode}.
\end{proof}

 \begin{rmk}  
  On peut montrer qu'il existe un morphisme $X \to C$
 fini \'etale qui est une forme tordue de la multiplication
 par 2 sur $J_{X}$ : voir \cite[p. 361]{XW18}  et
  la r\'ef\'erence dans la d\'emonstration
 de \cite[Prop. 4.7]{CV21}.
 \end{rmk}

 \begin{rmk}
 Sur tout corps, comme $[C]=2[X]$, si $X$ poss\`ede un point quadratique,
  alors $C$ poss\`ede un point rationnel. Ceci peut se voir facilement. S'il y a un point quadratique,
  on prend la droite qui passe par ce point et son conjugu\'e. Puis un autre $k$-point dessus.
  Il y a une forme $\lambda f + \mu g$ qui s'annule l\`a-dessus donc aussi sur la droite.
  Cette forme $\lambda f + \mu g$ est constitu\'ee de 2 hyperboliques, donc son d\'eterminant
  est un carr\'e. On a trouv\'e un point rationnel sur $y^2=det( \lambda f + \mu g)$.
 \end{rmk}

     \begin{prop}\label{intersectquelconqueP3}
     Soient $f(x_{0}, x_{1}, x_{2},x_{3} )$ et $g(x_{1}, x_{2},x_{3})$ deux formes
     quadratiques sur un corps $p$-adique. La $k$-vari\'et\'e $X \subset \P^3_{k}$
     d\'efinie  par $f=g=0$ poss\`ede un point dans une extension quadratique
     de $k$.
     \end{prop}
     \begin{proof}
     Comme le rang de la forme quadratique $g$ est au plus 3,
  le lemme \ref{intersecpasintegre},  qui vaut sur tout corps, r\'eduit la d\'emonstration
  au cas o\`u $g$ est de rang exactement 3 et $X$ est une  intersection compl\`ete lisse,
  donc une courbe g\'eom\'etriquement int\`egre.
   Comme $g$ est de rang 3, le polyn\^{o}me s\'eparable $det(\lambda f+ \mu g)$ a un z\'ero sur $k$,
     donc la courbe lisse $C$ d\'efinie par $y^2=det(\lambda f+ \mu g)$ a un point rationnel.
    Comme $k$ est $p$-adique, la proposition \ref{pseudoclass}(e), qui repose sur le th\'eor\`eme \ref{periode},  assure que $X$ poss\`ede un point dans
   une extension au plus quadratique de $k$. 
   \end{proof}
     
     \begin{rmk}
     Voici une variante de la d\'emonstration, dans le cas o\`u $X \subset \P^3_{k}$
     est une intersection compl\`ete lisse.
     Soit $D \subset \P^2_{k}$ la conique lisse d\'efinie par $g(x_{1}, x_{2},x_{3})=0$.
     On a la projection $p: X  \to D$ qui est un morphisme fini de degr\'e 2.
     L'application degr\'e sur les groupes de Picard g\'eom\'etriques induit
    des homomorphismes     $ \Pic(\cl{X})  \to \Z$
   et $ \Pic(\cl{D}) \to \Z$, et l'application $p^* : \Pic(\cl{D}) \to \Pic(\cl{X}) $
   induit la multiplication par $2$ sur $\Z$. 
     Comme $C$ est une conique, la fl\`eche $ \Pic(\cl{D}) \to \Z$ est un isomorphisme.
     Prenant les invariants sous l'action du groupe de Galois $\g={\rm Gal}(\k/k)$,
     on voit que $2$ est dans l'image de    $ \Pic(\cl{X})^\g  \to \Z$.
Ainsi la p\'eriode de $X$ divise 2. Comme $k$ est un corps $p$-adique,
le th\'eor\`eme \ref{periode} donne que l'indice de $X$ divise~$2$, et donc  $X$,
qui est une courbe de genre 1,
poss\`ede un point dans une extension quadratique de $k$.
      \end{rmk}
    
  \begin{rmk}\label{encoreunevariante}\label{variantebis}
    Voici une autre variante de la d\'emonstration, dans le cas o\`u $X \subset \P^3_{k}$
 est une intersection compl\`ete lisse. 
 Soit $C \to \P^1_{k}$ la projection. Soit $t=\mu/\lambda$
 et $\A^1_{k}=\Spec(k[t])$.
La quadrique g\'en\'erique $f+tg$ sur le corps $k(\P^1)=k(t)$
par passage \`a $k(C)$ a un d\'eterminant un carr\'e, et d\'etermine
un \'el\'ement $\alpha \in \Br(k(C))[2]$ qui appartient \`a $\Br(C)[2]$.

S'il existe un point $P$ de $C(k)$ tel que $\alpha(P)=0 \in \Br(k)$,
 comme $k$ est $p$-adique, 
on peut le supposer situ\'e au-dessus de $t_{0} \in k=\A^1_{k}(k)$, 
avec  $det(f+t_{0}g) \neq 0$.
La forme quadratique   $f+t_{0}g$, de rang 4 et de  d\'eterminant un carr\'e
est hyperbolique sur $k(P)=k$. La quadrique correspondante
contient un $\P^1_{k}$ qui coupe une autre quadrique du pinceau
en un point quadratique, donc $X$ poss\`ede un point sur
une extension quadratique.

Si au contraire pour tout point $P$ de $C(k)$, on  $\alpha(P)\neq 0$, alors
pour tout tel point, on a $\alpha(P)=\beta$ o\`u $\beta \in \Br(k)[2] \simeq \Z/2$
est l'unique \'el\'ement non nul. Alors la classe $\alpha - \beta \in \Br(C)$
s'annule en tout point $k$-rationnel de la courbe $C$, qui est de genre 1
et poss\`ede un $k$-point.
Par le th\'eor\`eme de dualit\'e  de Lichtenbaum  \cite[Thm. 4]{Li69} 
pour la courbe
elliptique $C$, qui repose sur le th\'eor\`eme de dualit\'e de Tate pour les vari\'et\'es ab\'eliennes sur un corps
$p$-adique, ceci implique que l'on a $\alpha -\beta= 0  \in \Br(C)$.
Pour toute extension quadratique $K/k$ du corps $p$-adique $k$, on a $\beta_{K}=0 \in \Br(K)$.
Ainsi, pour toute extension quadratique $K/k$, on a $\alpha_{K}= 0 \in \Br(C_{K}) \subset
\Br(K(C))$. Donc la forme quadratique $f+tg$  de rang 4 sur le corps $K(t)$ est hyperbolique sur son extension 
d\'eterminant $K(C)/K(t)$
et  par la proposition \ref{chatelet}  ceci implique que la forme  $f+tg$ est isotrope sur le corps $K(t)$.
Par la proposition \ref{amerbrumer}, ceci implique que $X$ a un $K$-point.
  \end{rmk}

     \begin{rmk}\label{curieux} 
 Soit $k$ $p$-adique. Supposons que $X \subset \P^3_{k}$ est une courbe lisse.
 Si dans le pinceau il existe une forme $g$ de rang 3, ou plus g\'en\'eralement si
 la courbe $C$
 d'\'equation $y^2=det(\lambda f + \mu g)$ 
 a un point rationnel,   d'apr\`es le th\'eor\`eme \ref{pseudoclass}(e)
 la courbe $X$ a un point quadratique.
D'apr\`es la proposition \ref{pointquaddroite}, 
 il existe  alors une   forme de rang 4 dans le pinceau $\lambda f + \mu g$ 
 qui est somme de  2 hyperboliques.    Cela montre qu'il existe $P \in C(k)$ tel que $\alpha(P)=0$. 
Le deuxi\`eme cas envisag\'e dans la remarque \ref{variantebis} ne peut donc pas se produire, mais il
faut utiliser le th\'eor\`eme de dualit\'e de Tate pour le voir.
     \end{rmk}

Pour les intersections compl\`etes {\it lisses}  de deux quadriques dans $\P^n_{k}$, $n \geq 4$,
le th\'eor\`eme suivant a \'et\'e d\'emontr\'e par  Creutz et Viray \cite[Thm. 1.2 (1)]{CV21}.

\begin{thm}\label{dansP4} 
 Soit $k$ un corps $p$-adique.
Soient $n\geq 4$ et  $X \subset \P^n_{k}$ une vari\'et\'e d\'efinie
par l'annulation de  deux formes quadratiques $f=g=0$. 
Alors $X$ poss\`ede un point dans    une extension quadratique de~$k$.
\end{thm}

\begin{proof}
Par intersection avec un espace lin\'eaire $\P^4_{k} \subset \P^n_{k}$
{\it quelconque}, il suffit de consid\'erer le cas $n=4$, ce qu'on suppose d\'esormais.
Si toute forme dans le pinceau $\lambda f + \mu g$ est de rang au plus 4,
alors $det(\lambda f + \mu g)$ est identiquement nul. Dans ce cas, le lemme \ref{intersecpasintegre} (ii)
assure qu'il existe  un point dans une extension quadratique.
Supposons donc qu'il existe une forme de rang 5 dans le pinceau.
Comme $k$ est un corps $p$-adique, toute forme quadratique  en 5 variables est isotrope.  
On est donc ramen\'e \`a un syst\`eme  
$$f(x_{0},x_{1},x_{2}) + x_{3}x_{4 }=0,  \hskip3mm  g(x_{0},x_{1},x_{2}, x_{3},x_{4})=0.$$
On coupe par $x_{4}=0$.
Notons $h(x_{0},x_{1},x_{2}, x_{3})=g(x_{0},x_{1},x_{2}, x_{3},0)$.
La proposition \ref{intersectquelconqueP3} assure que 
la vari\'et\'e $Y\subset \P^3_{k}$ d\'efinie par $f(x_{0},x_{1},x_{2})=h(x_{0},x_{1},x_{2}, x_{3})=0$
poss\`ede un point
quadratique. Il en est donc de m\^{e}me pour $X \subset \P^4_{k}$. 
\end{proof}

 \begin{rmk}
 Le cas $n \geq 6$ est facile : toute forme quadratique en $n+1 \geq 7$
 variables s'annule sur une droite d\'efinie sur $k$.
 L'intersection d'une telle droite avec une seconde quadrique
d\'efinit un point au plus quadratique sur l'intersection. Pour $n \geq 6$, on peut donner   cet
argument  sur tout corps $C_{2}$,
 ou sur un corps de nombres totalement imaginaire.
 \end{rmk}

 \begin{rmk}
 La  d\'emonstration de \cite[Thm. 2.1]{CV21} pour $X \subset \P^4_{k}$
  passe par une \'etude d\'etaill\'ee \`a la Hensel.
 Les auteurs donnent  une seconde d\'emonstration  \cite[Prop. 4.7]{CV21}, via des r\'esultats
 sur les courbes de genre 1.  Par cette m\'ethode ils obtiennent 
  un r\'esultat un peu plus faible lorsque $p=2$ : 
le pgcd des degr\'es des points ferm\'es  divise 2.
La d\'emonstration que nous avons donn\'ee est une variante 
 de cette seconde d\'emonstration. Elle vaut aussi pour $p=2$.
\end{rmk}

\begin{rmk}
Des exemples de surfaces de del Pezzo de degr\'e 4 sans point quadratique (et m\^eme d'indice 4) ont \'et\'e construits
par Wittenberg \cite[Rem. 7.8 (5)]{CV21}  et par Creutz-Viray   \cite[Thm. 7.6]{CV21} 
sur  des corps $C_{4}$. Sur des corps  $C_{3}$, des exemples ont \'et\'e construits par
 Koll\'ar  et par Ottem  \cite[Rem. 7.8 (4)]{CV21}. La question de savoir si une surface de del Pezzo de degr\'e~4
sur un corps $C_{2}$, ou
 sur un corps de nombres,  poss\`ede un point dans une extension quadratique
reste ouverte.  La question de savoir si une surface de del Pezzo de degr\'e 4 sur un corps $C_{2}$
a son indice qui divise 2 est aussi ouverte.
\end{rmk}

 \section{Intersections compl\`etes  lisses dans $\P^4_{k}$}\label{para4}

\begin{prop}\label{dansP4bis}
Soit $k$ un corps, ${\rm car}(k) \neq 2$. Soit  $X \subset \P^4_{k}$ une intersection compl\`ete lisse de deux quadriques, d\'efinie par $f=g=0$.

(i) Si $k$ est alg\'ebriquement clos, $X$ contient exactement 16 droites.
La forme g\'en\'erique $f+tg$ contient $2H$.

(ii) Si $k$ est fini, la forme g\'en\'erique $f+tg$ contient  $1H$.
 Toute forme non singuli\`ere $f+\lambda g$ contient $2H$.
 On a $X(k) \neq \emptyset$.
 
 (iii) Si $k$ est $p$-adique, $p \neq 2$, et $X$ a bonne r\'eduction comme intersection
 compl\`ete lisse de deux quadriques, alors $X(k)\neq \emptyset$.

 (iv) Si $k$ est $p$-adique,   toute forme  non singuli\`ere $f+\lambda g$ contient $1H$.
 La $k$-vari\'et\'e $X$ poss\`ede un point dans une extension quadratique de $k$.
Il  existe des valeurs de $\lambda \in k$ pour lesquelles $f+\lambda g$
 est non singuli\`ere et contient $2H$.  
 
 (v) Si $k$ est un corps de nombres, il existe une extension quadratique $K/k$
 telle que $X(\A_{K}) \neq \emptyset$.
 
 (vi) (Creutz-Viray)  Si $k$ est un corps de nombres, et si $X(\A_{k})\neq \emptyset$,
 alors l'indice de $X$ divise 2. Plus pr\'ecis\'ement il existe un point ferm\'e
 dont le degr\'e est dans l'ensemble $\{1, 2,  6, 10\}$.
\end{prop}

\begin{proof} 
 (i)  et (ii).  
 L'existence des 16 droites sur une surface de del Pezzo de degr\'e 4 sur un corps alg\'ebriquement clos
est un r\'esultat classique.   
 D'apr\`es la proposition \ref{corpsCi}, les autres \'enonc\'es valent pour toute intersection de deux quadriques dans $\P^4_{k}$ telle que $f+tg$ est non d\'eg\'en\'er\'ee.

 (iii) Ceci r\'esulte de (ii) par le lemme de Hensel.

(iv) Toute forme quadratique en 5 variables sur $k$  $p$-adique est isotrope. La seconde partie de 
 (iv) r\'esulte   du th\'eor\`eme \ref{dansP4} (cas lisse, Creutz-Viray),  assurant l'existence
 d'un point de $X$ dans une extension quadratique de $k$, et 
  de la proposition 
\ref{pointquaddroite}.

(v) Si $k$ est un corps de nombres, il existe un ensemble fini $S$ de places de $k$
tel que $X(k_{v})\neq \emptyset$ pour $v\notin S$. Pour chaque place $v \in S$,
le point (iv) donne une extension   $k_{v}(\sqrt{a_{v} })$ avec $a_{v} \in k_{v}^*$
telle que $X(k_{v}(\sqrt{a_{v}}) )\neq \emptyset$. Par approximation faible sur $k$,
on trouve $a\in k^*$ tel que sur $K=k(\sqrt{a})$ on ait $X(\A_{K})\neq \emptyset$.

(vi) La d\'emonstration qui suit est donn\'ee bri\`evement dans  \cite[Remark 4.8]{CV21}. 
Le polyn\^{o}me $det(\lambda f + \mu g)$ est de degr\'e 5. Il existe donc
une extension  $K/k$ de degr\'e 1, 3 ou 5 sur laquelle ce polyn\^{o}me 
a un z\'ero. Sur le corps $K$, on a donc une forme $g$ dans le pinceau
qui est de rang 4.  
L'hypoth\`ese $X(k_{v})$ non vide implique que les $k_{v}$-points
de $X$ sont Zariski denses. Ainsi la forme quadratique $g$ de rang 4  admet
des z\'eros non triviaux sur tous les compl\'et\'es de $K$. Par le principe de Hasse,
elle admet donc
un z\'ero non trivial sur $K$. Mais alors il existe une droite de $\P^1_{K}$
(passant par le sommet du c\^{o}ne) contenue dans la quadrique $g=0$.
Son intersection avec une autre quadrique du pinceau donne un point
de $X$ dans une extension au plus quadratique de $K$.
Ainsi $X$ poss\`ede un point  ferm\'e 
dont le degr\'e est dans $\{1, 2, 3, 5, 6, 10\}$.
D'apr\`es la proposition \ref{springer}, ceci se ram\`ene \`a
$\{1, 2, 6, 10\}$.
Comme $X$ clairement poss\`ede un
point ferm\'e de degr\'e dans $\{1,2,4\}$, on conclut que l'indice de $X$ divise 2.
\end{proof}

\begin{rmk}
Sur un corps de nombres, le th\'eor\`eme principal  \cite[Thm. 1.1]{CV21} de  Creutz  et Viray assure que
  pour une intersection  compl\`ete  lisse  $X$ de deux quadriques 
 dans $\P^n_{k}$, $n \geq 4$, l'indice de $X$ divise 2 \cite[Thm. 1.1]{CV21}, sans supposer $X(\A_{k}) \neq \emptyset$. La d\'emonstration
est beaucoup plus d\'elicate que celle des r\'esultats (v) et (vi) ci-dessus.
 Supposant certaines conjectures g\'en\'erales,  ils donnent  aussi des r\'esultats conditionnels \cite[Thm. 1.2 (4) (5))]{CV21}
 pour l'existence d'un point dans une extension quadratique. 
 \end{rmk}

\medskip

 Avant d'\'enoncer la proposition suivante, 
  faisons quelques rappels g\'en\'eraux sur les formes
 quadratiques.
 
Soit $k$ un corps,  ${\rm car}(k) \neq 2$. 
Soit $X \subset \P^4_{k}$ une intersection   de deux quadriques, donn\'ee
par un syst\`eme $f=g=0$. Soit
  $\varphi := f+tg$
  la forme quadratique g\'en\'erique sur $K=k(t)$.
 Supposons-la de rang 5. 
Soit $\psi= \varphi \perp < -det(\varphi)>$.

La forme $\psi $ est isotrope si et seulement si $\varphi$ repr\'esente $det(\varphi)$.
Supposons que c'est le cas.
Dans ce cas la forme $\varphi$ s'\'ecrit $det(\varphi)  \perp \varphi_{0}$, avec $\varphi_{0}$
de rang 4, de d\'eterminant 1, donc de la forme $det(\varphi) . <u,v,w,uvw>$ pour 
des $u,v,w \in K^*$.  On a donc
$$\varphi = det(\varphi) . < 1, u,v,w,uvw>.$$
Cette forme est semblable \`a une sous-forme de rang 5 de la forme de Pfister
$<<u,v,w>>$. La forme $\varphi$ est donc isotrope si et seulement si  la forme de Pfister 
$<<u,v,w>>$ est totalement hyperbolique.

\begin{prop}
Soit $k$ un corps de nombres et  $X \subset \P^4_{k}$ une intersection   de deux quadriques, donn\'ee
par un syst\`eme $f=g=0$.  Soit $\varphi =: f+tg$ la forme quadratique g\'en\'erique sur $K=k(t)$.
Supposons $\varphi$ non singuli\`ere, i.e. de rang 5.
Soit $\psi= \varphi \perp < -det(\varphi)>$. Cette forme est  semblable \`a une forme d'Albert.

Supposons que, pour toute place $v$ de $k$, on a  $X(k_{v}) \neq \emptyset$.

Les conditions suivantes sont \'equivalentes.

 (i)  $X(k) \neq \emptyset$.

 (ii) La forme $\varphi$ est isotrope sur le corps $k(t)$.
 
 (iii) La forme $\psi$
  est isotrope sur le corps $k(t)$.
 
 (iv) L'alg\`ebre de Clifford de $\psi$ est semblable  \`a une alg\`ebre de quaternions.
\end{prop}

\begin{proof}
 La proposition \ref{amerbrumer} donne
l'\'equivalence de (i) et (ii).
  Il suffit de montrer que (iii) implique (ii).
On a vu ci-dessus que sous l'hypoth\`ese (iii) on peut \'ecrire
$$\varphi = det(\varphi) . <1, u,v,w,uvw>.$$
Par l'hypoth\`ese $X(k_{v}) \neq \emptyset$,
 la forme $\varphi$ est isotrope sur $k_{v}(t)$.
Donc $<<u,v,w>>$ est totalement hyperbolique sur $k_{v}(t)$.
On sait  \cite[Prop. 1.1]{CTCS80} que l'application de restriction
de $k$ \`a chaque $k_{v}$ induit une injection sur les groupes de Witt :
$$W(k(t)) \hookrightarrow \prod_{v} W(k_{v}(t)).$$
On conclut que $<<u,v,w>>$ est totalement hyperbolique sur $k(t)$.
Et donc $\varphi$ est isotrope sur $k(t)$. Et donc $X(k)\neq \emptyset$
d'apr\`es la proposition \ref{amerbrumer}.
\end{proof}

\begin{rmk}
Comme on sait, il existe des contre-exemples au principe de Hasse
pour une surface de del Pezzo de degr\'e 4 (intersection compl\`ete lisse
de deux quadriques dans $\P^4_{k}$).
On conjecture qu'il n'en existe pas si l'on a $\Br(k)=\Br(X)$.
\end{rmk}

\section{Intersections compl\`etes lisses dans $\P^5_{k}$}\label{para5}

\subsection{Un premier point de vue}

\begin{prop}\label{dansP5} Soit $k$ un corps, ${\rm car}(k) \neq 2$. 
Soit  $X \subset \P^5_{k}$ une intersection compl\`ete lisse d\'efinie par $f=g=0$.

(i)
La vari\'et\'e $F_{1}(X)$ des droites contenues dans $X$ est 
un espace principal homog\`ene d'une vari\'et\'e ab\'elienne
de dimension 2.

(ii) Si $k$ est alg\'ebriquement clos,  
la forme g\'en\'erique $f+tg$ contient $2H$.

(iii) Si $k$ est  un corps fini $\F$,  de caract\'eristique $p\neq 2$,
  la vari\'et\'e $X$ contient au moins une droite $\P^1_{\F}$,
  et la forme g\'en\'erique   $f+tg$ contient  $2H$.
 Toute forme non singuli\`ere $\mu f+\lambda g$ contient $2H$.
 Si le corps fini  $\F$ a $q>30$   \'el\'ements,
  il existe au moins une forme  non singuli\`ere 
 $  f+\lambda g$  dans le pinceau sur $\F$  qui  est totalement hyperbolique,
  et alors $X$ contient une conique.

 (iv) Si $k$ est $p$-adique et $p>2$, et  si  $X$ a bonne r\'eduction comme intersection
 compl\`ete lisse de deux quadriques, alors $X$ contient des droites $\P^1_{k}$, 
 la forme g\'en\'erique $f+tg$ contient  $2H$,
 et toute forme  non singuli\`ere $f+\lambda g$, $\lambda \in k$, contient $2H$.
 Si en outre le cardinal $q$ du corps fini  r\'esiduel satisfait $q>30$, 
  alors il existe des formes non singuli\`eres $f+\lambda g$ qui s'annulent sur un $\P^2_{k}$,
 autrement dit sont totalement hyperboliques, et donc $X$ contient une conique.

  (v) Si $k$ est $p$-adique,   
  toute forme  non singuli\`ere $f+\lambda g$ contient $1H$.
Il  existe des valeurs de $\lambda \in k$ pour lesquelles $f+\lambda g$
 est non singuli\`ere et contient $2H$.

   \end{prop}

\begin{proof}
 L'\'enonc\'e (i) (voir Prop. \ref{generalintegre} (b)) est \'etabli 
dans \cite[Thm. 1.1]{XW18}.
 
 Pour l'\'enonc\'e (ii), d'apr\`es
  la proposition \ref{corpsCi}, il vaut pour toute intersection  compl\`ete de deux quadriques dans $\P^5_{k}$
 telle que la forme $f+tg$ soit non d\'eg\'en\'er\'ee.

L'\'enonc\'e (i) et
  le th\'eor\`eme de Lang  sur les espaces principaux homog\`enes de groupes alg\'ebriques connexes
 sur un corps fini  donnent le premier \'enonc\'e de (iii).
  Consid\'erons 
 la courbe affine  lisse $C/\F $  
d'\'equation  $y^2=-det(f+ \lambda g) $. Soit $D \to \P^1$ 
l'extension de $C \to \A^1_{\F}$ (donn\'e par $\lambda$)
\`a un   rev\^{e}tement double de $\P^1_{\F}$
ramifi\'e en 6 points g\'eom\'etriques. La  courbe projective lisse $D$
 est de genre 2.
Il y a au plus 6 points rationnels de $D$ au-dessus du lieu de ramification
et au plus 2 points rationnels au-dessus du point \`a l'infini de $\P^1$.
  L'estimation de Weil sur le nombre $N_{D}(\F)$
de points rationnels de la courbe $D$ donne $N_{D}(\F) \geq   1+ q -4  \sqrt{q}$.
Pour $q >30$, on trouve donc un point de $C(\F)$  d'image $\lambda \in \A^1(\F)$
tel que   la forme quadratique $   f+ \lambda g$  sur le corps fini $\F$ soit 
 de rang 6 et  de d\'eterminant l'oppos\'e d'un carr\'e, et donc totalement hyperbolique.

Pour \'etablir (iv), on utilise le fait que si $X/k$ a bonne r\'eduction comme intersection lisse
de deux quadriques, alors la $k$-vari\'et\'e $F_{1}(X)$ s'\'etend en un sch\'ema projectif
et lisse sur l'anneau des entiers de $k$. L'\'enonc\'e (iii) et le lemme de Hensel donnent
alors que $X$ contient une droite $\P^1_{k}$. Le dernier \'enonc\'e de (iv) r\'esulte du
dernier \'enonc\'e de (iii) : il suffit de relever une forme non singuli\`ere $f+\lambda g$
du corps fini sur le corps $p$-adique.

 Pour (v), on consid\`ere une section  hyperplane lisse $Y = X \cap \Pi$.
D'apr\`es la proposition \ref{dansP4bis} (iv), $Y$ poss\`ede un point dans une extension
au plus quadratique de $k$. Il en est donc de m\^eme de $X$.
D'apr\`es la proposition \ref{pointquaddroite}, il existe alors des
quadriques non singuli\`eres 
dans le pinceau qui contiennent une droite de $\P^5_{k}$,  i.e.
il existe   $\lambda \in k$ tel que
$f+\lambda g$ soit de rang maximal et contienne $2H$. 
\end{proof}

On peut maintenant donner une d\'emonstration inconditionnelle de l'\'enonc\'e
 \cite[Thm. 1.2 (2) (a)]{CV21} de Creutz et Viray.
 
 \begin{thm}\label{quadratiqueP5nombres}
 Soient $k$ un corps de nombres et $X \subset \P^n_{k}$ 
 une intersection compl\`ete  lisse de deux quadriques
 donn\'ee par $f=g=0$.
 Pour $n \geq 5$, il existe un point quadratique sur $X$.
 \end{thm}
 
\begin{proof}  Supposons d'abord $n=5$.
  Hors d'un ensemble fini $S$ de places de $k$, contenant les places $2$-adiques,
 l'intersection lisse de deux quadriques $X$ a bonne r\'eduction
 comme intersection de deux quadriques.
 
 Pour toute place $v$ de $k$  hors de $S$, par la proposition \ref{dansP5}(iv)  
 toute forme non singuli\`ere
  $\lambda_{v} f + \mu_{v} g $ dans le pinceau contient $2H$.

 D'apr\`es le th\'eor\`eme local de Creutz-Viray \cite[Thm. 1.2 (1)]{CV21} (cas lisse du th\'eor\`eme
 \ref{dansP4} ci-dessus),  pour toute place $v$ de $k$,
 $X$ contient un point quadratique  sur $k_{v}$
 (c'est clair pour $k_{v}=\R$),
 donc par la proposition \ref{pointquaddroite}
  il existe une forme quadratique non singuli\`ere
 $\lambda_{v} f + \mu_{v} g $ dans le pinceau
 qui contient $2H$.
 Par le th\'eor\`eme des fonctions implicites, ceci sera encore
 vrai pour tout point de $\P^1(k_{v})$ proche de $(\lambda_{v},\mu_{v})$.
 
 Par approximation faible sur $\P^1_{k}$, on trouve donc 
une forme non singuli\`ere
   $\lambda f + \mu g$ dans le pinceau
 sur $k$ qui localement pour toute place contient $2H$.
 Par le th\'eor\`eme \ref{Hasse}, la forme contient alors
$2H$  sur le corps $k$.  Par l'implication \'evidente
 dans la   proposition \ref{pointquaddroite}, on conclut que
  $X$ poss\`ede un point quadratique.  
 
  \smallskip
 
 Pour $n \geq 6$, on peut proc\'eder par r\'eduction au cas $n =5$, mais
 il est plus simple d'observer que,  comme toute forme non singuli\`ere de
 rang 7 sur un corps $p$-adique contient $2H$,
 l'argument final ci-dessus permet  facilement de prouver l'existence
  d'une forme $\lambda f + \mu g$
 non singuli\`ere contenant $2H$. 
\end{proof}

 \begin{rmk}
Pour cette d\'emonstration, il suffit d'utiliser le fait que $X$  et
 la $k$-vari\'et\'e $F_{1}(X)$ s'\'etendent en des mod\`eles projectifs et lisses
 au-dessus d'un ouvert non vide du spectre de  l'anneau des entiers du corps
 de nombres $k$, les fibres sur les corps r\'esiduels $\kappa(v)$ 
  \'etant  les vari\'et\'es g\'eom\'etriquement int\`egres $F_{1}(X_{\kappa_{v}})$
attach\'ees aux intersections lisses de deux quadriques dans $\P^5_{\kappa(v)}$.
\end{rmk}
 
  \begin{rmk}  
  Creutz et Viray \cite[Thm. 1.1]{CV21} ont montr\'e que pour~$k$ un corps de nombres et $X \subset \P^4_{k}$ une
  intersection compl\`ete lisse de deux quadriques, l'indice $I(X)$ divise 2. La question  \cite[Question 1.3]{CV21}
  si une telle vari\'et\'e $X$ poss\`ede un point quadratique est ouverte.
  \end{rmk}

 \bigskip
 
Le th\'eor\`eme suivant a \'et\'e \'etabli par Iyer et Parimala \cite{IP22}. Nous en offrons une
d\'emonstration alternative.

\begin{thm}\label{iyerparimala} (Iyer et Parimala)
Soient $k$ un corps de nombres et $X \subset \P^5_{k}$ une intersection compl\`ete
lisse de deux quadriques donn\'ee par un syst\`eme $f=0, g=0$.
Supposons que la courbe $C$ d'\'equation 
$y^2=-det(\lambda f + \mu g)$ a un point de degr\'e impair,
ce qui est le cas par exemple s'il existe une forme de rang 5 dans le pinceau.
Si pour chaque place $v$, la vari\'et\'e $X$ contient une $k_{v}$-droite,
alors $X$ a un point rationnel.
\end{thm}

\begin{proof}  
Par un lemme de d\'eplacement bien connu combin\'e \`a la proposition \ref{springer}
on se ram\`ene au cas o\`u $C$ poss\`ede un point rationnel satisfaisant $y \neq 0$.
Donc il existe une forme non d\'eg\'en\'er\'ee $\lambda_{0} f + \mu_{0} g$ dans le pinceau
sur $k$
dont le d\'eterminant est  l'oppos\'e d'un carr\'e. 
 Par hypoth\`ese, sur chaque compl\'et\'e $k_{v}$,
la forme $\lambda_{0 }f + \mu_{0} g$ contient deux hyperboliques.
Comme son d\'eterminant est l'oppos\'e d'un carr\'e, elle est
totalement hyperbolique.
Puisque cela vaut sur chaque $k_{v}$, par le th\'eor\`eme \ref{Hasse}
 $\lambda_{0 }f + \mu_{0} g$ est totalement hyperbolique sur $k$.
 La vari\'et\'e $X  \subset \P^5_{k}$ contient donc une conique, 
 et elle a des points dans tous les compl\'et\'es.
D'apr\`es le th\'eor\`eme \ref{sal93} (Salberger), elle  a un $k$-point. 
 \end{proof}

\subsection{Un deuxi\`eme point de vue}

 Soient $k$ un corps de caract\'eristique diff\'erente de 2, et $X \subset \P^5_{k}$
 une intersection compl\`ete lisse de deux quadriques, d\'efinie par $f=g=0$.
 Soit $C$ la courbe lisse d\'efinie par l'\'equation
 $$y^2=-det(\lambda f + \mu g).$$
 
 \bigskip
 
On renvoie \`a  Wang  \cite{XW18}  pour les faits suivants.
 
 Soit $J=\Pic^0_{C/k}$.  C'est une vari\'et\'e ab\'elienne sur $k$.
 Soit $P=\Pic^1_{C/k}$. 
 C'est un espace principal homog\`ene sous $J$.
 Il est trivial si $C$ poss\`ede un point ferm\'e de degr\'e impair.
  Soit $F_{1}=F_{1}(X)$ la vari\'et\'e
 des droites contenues dans $X$. 
 C'est un espace principal homog\`ene sous $J$.
  Dans $H^1(k,J)$, on a les  \'egalit\'es
 $[P] = 2 [F_{1}]$ et $2[P]=0$. Donc $4[F_{1}]=0$.

\bigskip

Les propri\'et\'es suivantes sont discut\'ees dans   la litt\'erature
\cite{Reid72}  \cite[\S 1]{ABB14}, et dans l'appendice au pr\'esent article
(propositions \ref{prop:gr} et \ref{prop:sb}).

On consid\`ere la vari\'et\'e $G_{2}(X)$ form\'ee des couples
$(H, Q)$ o\`u $Q $ est une quadrique de $\P^5_{k}$ contenant $X$, et o\`u
$H \subset Q \subset \P^5$ est un  espace lin\'eaire de dimension 2 de $\P^5$
contenu dans la quadrique $Q$.
On a une projection \'evidente $G_{2}(X) \to \P^1_{k}$ associant
\`a $(H,Q)$ le param\`etre de $Q$.
On montre 
(\cite[Thm. 1.10]{Reid72} \cite[Prop. 1.18]{ABB14}, appendice A ci-dessous)
que la factorisation de Stein de ce morphisme est
$$  G_{2}(X) \to C \to \P^1_{k}$$
avec $C$ la courbe ci-dessus, 
et que le morphisme $G_{2}(X) \to C$ d\'efinit
un sch\'ema de Severi-Brauer (\`a fibres lisses) de dimension relative 3,
associ\'e \`a une classe $\alpha_{X} \in \Br(C)[2]$ dont l'image dans
$\Br(k(C))$ est d'indice  divisant 4. 
Elle est donc la classe d'une alg\`ebre de biquaternions sur $k(C)$.

  \medskip

On dispose par ailleurs de la forme quadratique
$q:= f+tg$ sur le corps $k(t)$, qui sur l'extension $k(C)$ de $k(t)$ acquiert bonne
r\'eduction partout. 
La forme quadratique $q_{k(C)}$ est une forme quadratique de rang 6
de d\'eterminant~$-1$.  

\medskip
 
 Rappelons que sur un corps $F$ de caract\'eristique diff\'erente de 2, toute forme quadratique $\phi$
 de rang 6 et de d\'eterminant sign\'e 1, donc de d\'eterminant $(-1)$,
 est appel\'ee une forme d'Albert.   
  Elle est   un multiple scalaire d'une forme du type
 $$ <-a,-b, ab , -c, -d, cd>.$$ 
Son invariant de Clifford $c(\phi) \in \Br(F)[2]$ est 
la somme des deux classes d'alg\`ebres de quaternions
  $(a,b)$ et $(c,d)$.
 La forme $\phi$  est isotrope si et seulement si l'indice de $c(\phi)$ divise 2.
 La forme $\phi$  est totalement hyperbolique si et seulement si $c(\phi)= 0$.
 Pour tout ceci, et la d\'emonstration de la proposition \ref{identif} ci-dessus, voir 
 voir \cite[8.1.3, 8.1.4, 8.1.6]{Ka08} et 
 la  Proposition B3 de
l'appendice B de \cite{ABB14}.

 \begin{prop}\label{identif}
{\it La classe $\alpha_{X} \in \Br(C)$ est la classe de l'alg\`ebre
de Clifford associ\'ee \`a  la forme quadratique  $q_{k(C)}$.}
\end{prop}

 \begin{prop}
Soient  $X \subset \P^5_{k}$, $C$ et  $\alpha_{X} \in \Br(C) \subset \Br(k(C))$  comme ci-dessus.

(i)  Si l'on a $X(k) \neq \emptyset$, c'est-\`a-dire si la forme quadratique 
$f+tg$ a un z\'ero sur le corps $k(t)$,
alors l'indice de $\alpha \in \Br(k(C))$ 
 divise 2.
 
 (ii) Si l'indice de $\alpha \in \Br(k(C))$  divise 2, alors 
 la forme quadratique $f+tg$ a un z\'ero sur l'extension quadratique $k(C)$ de $k(t)$.
  \end{prop}
\begin{proof}
    Via la proposition \ref{identif}, ceci r\'esulte
   des rappels sur les formes d'Albert donn\'es ci-dessus : l'indice de $\alpha \in \Br(k(C))$
   divise 2 si et seulement si la forme $f+tg$ est isotrope sur le corps $k(C)$. 
   Pour (i), voir aussi la d\'emonstration g\'eom\'etrique \cite[\S4, Cor. 7]{HT21}.
    \end{proof}

 \begin{thm}\label{alphaetpoint}
Pour  $X \subset \P^5_{k}$, $C$ et  $\alpha_{X}$  comme ci-dessus, on a \'equivalence
entre les propri\'et\'es suivantes :

(i) Il existe une quadrique non singuli\`ere dans le pinceau qui contient un 
 $\P^2_{k}$.

(ii)  Il y a une forme quadratique non singuli\`ere totalement hyperbolique dans le
pinceau  $\lambda f+\mu g$.

(iii) Il y a un $k$-point $m \in C(k)$ 
non ramifi\'e pour $C \to \P^1_{k}$
tel que  $\alpha_{X}(m)=0 \in \Br(k)$.

Elles impliquent que $X$ contient une conique.
 \end{thm}
 
\begin{proof}  Les \'enonc\'es (i) et (ii) sont clairement \'equivalents.
Comme rappel\'e ci-dessus et dans la Proposition
\ref{kuznestov}, le morphisme $G_{2}(X) \to \P^1_{k}$
se factorise par $C \to \P^1_{k}$, o\`u la courbe $C$
est donn\'ee par $y^2=-det(\lambda f + \mu g)$. 
Un $k$-point $m$ de $C$ non dans le lieu de ramification
a pour image un point $(\lambda,\mu) \in \P^1(k)$
tel que $\lambda f+\mu g$ est de rang maximal,
et qui contient un $\P^2_{k}$ si et seulement si 
 $\alpha_{X}(m)=0 \in \Br(k)$.
\end{proof}

\bigskip

Une d\'emonstration du th\'eor\`eme suivant (avec la restriction  $X(k)\neq \emptyset$)
est esquiss\'ee dans
  \cite[Theorem 9]{HT21}. Nous  proposons une d\'emonstration alternative.

 \smallskip
 
 \begin{thm}\label{alpha}
 Soient $X \subset \P^5_{k}$,  $C$ et $\alpha_{X} \in \Br(C)$ comme ci-dessus.

  (i)  Si $X$ contient une droite sur $k$,  alors  $\alpha_{X}=0 \in \Br(k(C))$.
 
 (ii) Si $X(k)\neq \emptyset$ et $\alpha_{X}=0 \in \Br(k(C))$, alors  $X$ contient une droite sur $k$.
  \end{thm}
 
\begin{proof}
 
  (i)  Si $X$ contient une droite $\P^1_{k}$, alors
la forme quadratique $f+tg$ sur  $k(t)$ contient $2H$. 
Donc  $$q \simeq <1,-1> \perp <1,-1> \perp \rho. <1,det(q)>.$$ 
Sur l'extension quadratique  $k(C)$ cette forme est totalement hyperbolique.
Donc $\alpha_{X}=0 \in \Br(k(C))$.

(ii)  L'hypoth\`ese $X(k) \neq \emptyset$ implique que
 $q= f+tg$ contient $1H$.  On a donc
 $$q \simeq <1,-1> \perp q_{1}$$
 sur $k(t)$, avec $q_{1}$ de d\'eterminant un carr\'e dans $k(C)^{\times}$.
 L'hypoth\`ese
$\alpha_{X}=0 \in  \Br(k(C))$  implique alors que cette forme quadratique
   $q_{1}$ est compl\`etement hyperbolique sur $k(C)$. Puisque    $k(C)/k(t)$ est
   l'extension discriminant et $q_{1}$ a rang 4, par la proposition \ref{chatelet}
   ceci implique que
 $q_{1}$ est isotrope sur $k(t)$.
Donc $q$ contient $2H$ sur $k(t)$.
La quadrique g\'en\'erique $f+tg=0$ contient donc une droite. 
Par
la proposition
  \ref{amerbrumer}, la vari\'et\'e d'\'equations
 $f=g=0$ 
 contient un $\P^1_{k}$.
\end{proof}
 
 \medskip

L'hypoth\`ese $X(k)\neq \emptyset$ n'est pas n\'ecessaire. On a :
 
   \begin{thm}\label{alphabis}
Soient $X \subset \P^5_{k}$, $C$ et $\alpha_{X} \in \Br(C)$ comme ci-dessus.
Soit $J=\Pic^0_{C/k}$.
Les conditions suivantes sont \'equivalentes :
 
  (i)  La vari\'et\'e $X$ contient une droite sur $k$.
  
  (ii)  On a   $\alpha_{X}=0 \in \Br(k(C))$.
  
  (iii) La classe de  $F=F_{1}(X)$  dans  $H^1(k,J)$ est nulle.
\end{thm}

\begin{proof} On sait \cite{XW18} que la $k$-vari\'et\'e $F_{1}(X)$ des droites de $X$
est un espace principal homog\`ene de $J$, qui a donc une classe $[F]$  dans $H^1(k,J)$.
Les \'enonc\'es (i) et (iii) sont \'equivalents. D'apr\`es le th\'eor\`eme \ref{alpha},
(i) implique (ii).  Supposons (ii). Soit $k(X)$ le corps des fonctions de $X$.
 Consid\'erons $X\times_{k}k(X)$. D'apr\`es le th\'eor\`eme \ref{alpha},
$X\times_{k} k(X)$ poss\`ede une droite. L'image de $[F]  \in H^1(k,J)$  dans  $H^1(k(X),J)$
est donc nulle. Mais l'application $ H^1(k,J) \to H^1(k(X),J)$ est injective,
car toute application $k$-rationnelle de la $k$-vari\'et\'e 
g\'eom\'etri\-quement rationnellement connexe $X$ vers un espace
  principal homog\`ene d'une vari\'et\'e ab\'elienne est constant. 
Ainsi $[F]=0 \in H^1(k,J)$, et $X$ contient une droite sur $k$.
\end{proof}

 \begin{rmk}
 Si $X$ contient une droite sur $k$, alors $X$ est $k$-rationnelle. La r\'eciproque
 est un th\'eor\`eme d\'elicat qui fut prouv\'e par Benoist et Wittenberg, apr\`es 
 un travail pr\'eliminaire de Hassett et Tschinkel.
  \end{rmk}

  Au th\'eor\`eme  \ref{iyerparimala} on a donn\'e une d\'emonstration
  alternative d'un th\'eor\`eme de  Iyer et Parimala  \cite{IP22}.
Voici une autre variante de cette d\'emonstration.

  \begin{thm}\label{autreIP}
  Soient $k$ un corps de nombres et  $X \subset \P^5_{k}$ une intersection compl\`ete lisse de deux quadriques.
  Soit $C$ la courbe projective
 et lisse d\'efinie par  l'\'equation $y^2=-det(\lambda f + \mu g)$.
  Si $X$ poss\`ede une droite sur chaque $k_{v}$, et si l'indice de la courbe $C$ est 1,
  alors   $X$ poss\`ede un point rationnel.
  \end{thm}
\begin{proof} D'apr\`es le th\'eor\`eme \ref{alpha}, 
$\alpha_{X} \in \Br(C)$ a son image nulle dans $\Br(C_{k_{v}})$ pour
chaque place $v$ de $k$. Par ailleurs il existe un point
$P$ ferm\'e de degr\'e impair sur $C$, qu'on peut prendre en dehors
du lieu de ramification de $C \to \P^1_{k}$. Soit $K=k(P)$ son corps r\'esiduel.
La classe $\alpha_{X}(P) \in \Br(K)$ s'annule dans tous les compl\'et\'es
$K_{w}$ de $k$. D'apr\`es le principe de Hasse pour le groupe de Brauer
d'un corps de nombres, $\alpha_{X}(P) =0$.
D'apr\`es le th\'eor\`eme \ref{alphaetpoint}, il existe  une conique
  dans $X_{K}$.  
Comme $X(\A_{k}) \neq \emptyset$ et donc $X(\A_{K})\neq \emptyset$,
et que $X_{K} \subset \P^5_{K}$ contient une conique,
le th\'eor\`eme   \ref{sal93}  (Salberger) donne $X(K) \neq \emptyset$ et donc,
par   la Proposition \ref{springer}, $X(k) \neq \emptyset$.
\end{proof}

\section{Intersections compl\`etes   lisses dans  $\P^6_{k}$}\label{para6}

Ce cas est assez myst\'erieux.

\begin{prop}\label{dansP6} Soit $k$ un corps, ${\rm car}(k) \neq 2$. 
Soit  $X \subset \P^6_{k}$ une intersection compl\`ete lisse  de deux quadriques d\'efinie par $f=g=0$.

(i) Supposons ${\rm car}(k)=0$. La vari\'et\'e  $F_{1}(X)$ est 
une vari\'et\'e de Fano projective lisse, g\'eom\'etriquement rationnelle, de dimension 4,
dont le groupe de Picard g\'eom\'etrique est libre de rang~8. 

 (ii) Si $k$ est alg\'ebriquement clos,  
la forme g\'en\'erique $f+tg$ contient $3H$.
La vari\'et\'e $X$ contient $64$ plans $\P^2$.

(iii) Si $k$ est un corps fini $\F$,  la forme g\'en\'erique $f+tg$ contient  $2H$.
 Toute forme non singuli\`ere $f+\lambda g$ contient $3H$.
 La vari\'et\'e $X$ contient au moins une droite $\P^1_{\F}$.
 
 (iv) Si $k$ est $p$-adique 
 et $p>2$,
  et  si $X$ a bonne r\'eduction comme intersection
 compl\`ete lisse de deux quadriques, alors $X$ contient des droites $\P^1_{k}$,
  la forme g\'en\'erique $f+tg$  contient $2H$.
Si  de plus le cardinal du corps fini  r\'esiduel est au moins 9,  alors il existe des
 formes non singuli\`eres $f+\lambda g$ qui s'annulent sur un $\P^2_{k}$,
 autrement dit contiennent $3H$. Dans ce cas $X$ 
 contient une conique d\'efinie sur $k$.

 (v) Si $k$ est $p$-adique,  toute forme  non singuli\`ere $f+\lambda g$ contient $2H$.
  
   (vi) Si $k$ est un corps de nombres, il existe des formes non singuli\`eres 
   $f+\lambda g$ qui contiennent $2H$, et
       $X$ contient un point sur une extension quadratique.
 \end{prop}

\begin{proof}
L'\'enonc\'e (i) est mis pour m\'emoire, voir la proposition \ref{generalintegre} (b).
La premi\`ere partie de l'\'enonc\'e (ii), et l'existence d'au moins un plan $\P^2_{k}$ dans $X$
r\'esultent de la proposition \ref{corpsCi}. 
La  vari\'et\'e $F_{2}(X)$ est  lisse  de dimension z\'ero d'apr\`es la
proposition \ref{generalintegre}  (a). Elle a 64 points (voir la d\'emonstration
du Lemme \ref{lem:fr} de l'appendice).

Pour (iii), comme le corps $\F(t)$ est $C_{2}$, la forme  $f+tg$ qui est non d\'eg\'en\'er\'ee et
de rang 7
contient $2H$. La proposition \ref{amerbrumer} donne alors que $X$
contient une droite $\P^1_{\F}$.

Pour (iv), on aimerait trouver une section  par un bon hyperplan $\Pi$  telle que 
$X \cap \Pi$ ait  bonne r\'eduction, et  appliquer
la proposition \ref{dansP5}.
Quitte \`a remplacer $\F$ par une extension finie  $\F'/\F$ de degr\'e impair, 
et $k$ par l'extension non ramifi\'ee $k'/k$ correspondante,
on peut trouver une telle section.  On trouve ainsi une extension $k'/k$ de degr\'e impair
telle  que $X \cap \Pi$ et donc $X$ contienne une droite $\P^1_{k'}$.
 La forme $f+tg$ contient donc $2H$ sur $k'(t)$. Par une variante du th\'eor\`eme
de Springer, elle contient $2H$  sur $k(t)$. Et donc, par la proposition \ref{amerbrumer}, la vari\'et\'e 
$X$ contient une droite $\P^1_{k}$.
Pour le dernier point de (iv), notons   que l'on peut trouver $\lambda, \mu$ dans le corps r\'esiduel 
fini $\F$ tels que $\lambda f + \mu g$  soit de rang $7$ sur le corps fini  et contienne $3H$. Ceci se rel\`eve sur le corps $p$-adique.

Pour (v), il suffit de remarquer que toute forme quadratique non d\'eg\'en\'er\'ee
de rang 7 contient $2H$.

Pour (vi), voir le th\'eor\`eme \ref{quadratiqueP5nombres} -- dont la d\'emonstration
dans le cas $n\geq 6$  est   simple.
\end{proof}

\section{Intersections compl\`etes lisses dans  $\P^7_{k}$}\label{para7}

 \begin{prop}\label{dansP7}
  Soit $k$ un corps, ${\rm car}(k) \neq 2$.  Soit $X \subset \P^7_{k}$ une intersection compl\`ete lisse de deux quadriques
d\'efinie par $f=g=0$.

(i) Suppons ${\rm car}(k)=0$. La vari\'et\'e $F_{1}(X)$ des $\P^1$ contenus dans $X$ 
est projective et lisse et g\'eom\'etri\-quement connexe.
C'est une vari\'et\'e  de Fano g\'eom\'e\-tri\-quement rationnelle,
de groupe de Picard g\'eom\'etrique $\Z$.

(ii) La vari\'et\'e  $F_{2}(X)$ des $\P^2$ contenus dans $X$ est 
un espace principal homog\`ene d'une vari\'et\'e ab\'elienne
de dimension 3.

 (iii) Si $k$ est alg\'ebriquement clos,  
la forme g\'en\'erique $f+tg$ contient $3H$.

(iv) Soit $k=\F$   un corps   fini.  
 Toute forme non singuli\`ere $f+\lambda g$ contient $3H$.
 La vari\'et\'e $X$ contient au moins un $\P^2_{\F}$. 
 La forme g\'en\'erique $f+tg$ contient  $3H$.

 (v) Si $k$ est $p$-adique
 et $p>2$, 
 et $X$ a bonne r\'eduction comme intersection
 compl\`ete lisse de deux quadriques, alors $X$ contient des $\P^2_{k}$, 
 la forme g\'en\'erique $f+tg$ contient $3H$,
 et toute forme non singuli\`ere $f+\lambda g$ contient $3H$.

  (vi) Si $k$ est $p$-adique, toute forme non singuli\`ere $f+\lambda g$ contient $2H$.
  
  (vii) Si $k$ est $p$-adique, et s'il existe une forme de rang 7 dans le pinceau $\lambda f + \mu g$,
  alors il existe une forme de rang 8 qui contient $3H$.
  
  (viii) Si $k$ est un corps de nombres, et s'il existe une forme de rang 7 dans le pinceau $\lambda f + \mu g$,
  alors il existe une forme de rang 8  dans le pinceau qui contient $3H$.
  
  (ix)   Si $k$ est un corps de nombres, et  s'il existe une forme de rang 7 dans le pinceau $\lambda f + \mu g$,
  et si $X(\A_{k})\neq \emptyset$, alors $X$ poss\`ede un point rationnel.
 \end{prop}
  
\begin{proof}   Les \'enonc\'es (i) et (ii) ont \'et\'e donn\'es  
  \`a la proposition \ref{generalintegre} (d) (e) en caract\'eristique z\'ero.
  Le fait que (ii) vaille pour $F_{2}(X)$ en caract\'eristique 
  diff\'erente de 2 est \'etabli dans \cite{XW18}.
  
  (iii)  C'est un cas particulier de la proposition \ref{corpsCi}.

  (iv) Sur un corps fini $k=\F$, toute forme quadratique de rang au moins 3
  est isotrope. Tout espace principal homog\`ene sous une vari\'et\'e ab\'elienne sur $\F$ est trivial.
 De (ii) on d\'eduit par le th\'eor\`eme de Lang que l'on a $F_{2}(X)(\F)
  \neq \emptyset$. Donc la forme g\'en\'erique $f+tg$ contient $3H$.

(v) Si $k$ est $p$-adique 
et $p>2$,
et si $X$ a bonne r\'eduction comme intersection compl\`ete lisse
de deux quadriques, alors il en est de m\^eme
 de $F_{2}(X)$. 
 Par (iv) et le lemme de Hensel,
$F_{2}(X)$ a un $k$-point. Ainsi $X$
contient un $\P^2_{k}$ et donc toute forme non singuli\`ere
dans le pinceau contient $3H$.

(vi)  Sur un corps $p$-adique, et aussi sur un corps $C_{2}$,
toute forme quadratique en au moins 7 variables contient $2H$.

(vii) On peut supposer que le polyn\^ome  s\'eparable $P(t)= det(f+tg)$ est
 de degr\'e 8. Par hypoth\`ese, il admet un z\'ero sur $k$, donc un
z\'ero simple. On voit alors facilement qu'il existe $\lambda \in k$
tel que $P(\lambda) \in k$ n'est pas un carr\'e. La forme
$f +\lambda g$ est alors une forme de rang 8 dont le d\'eterminant  $\delta $ n'est 
pas un carr\'e. Comme le corps est $p$-adique, la forme s'\'ecrit
$$ <1,-1> \perp <1,-1> \perp q,$$
avec $q$ une forme quadratique de rang 4 de d\'eterminant    non carr\'e dans $k$.
Une telle forme est isotrope (Proposition \ref{bienconnu}).

(viii) et (ix)  En combinant (v), (vii)  et l'approximation faible,
on trouve $\lambda$ dans le corps de nombres $k$  tel que $f+\lambda g$ est de rang 8 et contient
$3H$ sur chaque $ k_{v}$, donc aussi sur $k$ par le th\'eor\`eme \ref{Hasse}. Ainsi $X$ contient une conique,
et sous l'hypoth\`ese $X(\A_{k}) \neq \emptyset$ le th\'eor\`eme  \ref{sal93} (Salberger) donne
$X(k) \neq \emptyset$. 
\end{proof}

\begin{rmk}
(a) Pour la d\'emonstration dans le cas global, au lieu d'utiliser  (v),
il suffit   de  savoir que
pour presque tout compl\'et\'e $k_{v}$, on a $F_{2}(X)(k_{v}) \neq \emptyset$.
Ceci r\'esulte simplement du fait que la $k$-vari\'et\'e $F_{2}(X)$ est
g\'eom\'etriquement int\`egre (Proposition \ref{generalintegre}(i)).

(b) Sur tout corps $k$, l'hypoth\`ese qu'il existe une forme de rang 7 dans
le pinceau \'equivaut au fait que le polyn\^ome homog\`ene
 $det(\lambda f + \mu g)$ poss\`ede un z\'ero sur $k$  (par n\'ecesssit\'e z\'ero simple, puisque $X$
 est lisse).
 
 (c) Pour obtenir la conclusion de (ix), il suffirait de supposer que 
$det(\lambda f + \mu g)$ poss\`ede un z\'ero dans une extension de degr\'e impair.
\end{rmk}

\bigskip

 Sur un corps $p$-adique $k$, sous les hypoth\`eses suppl\'ementaires que $X(k)$ est non vide
et que le cardinal du corps r\'esiduel  est au moins \'egal \`a 32,
Heath-Brown \cite[Thm. 2]{HB18}  montra que pour $X \subset \P_{k}^7$
une intersection compl\`ete lisse
il existe une quadrique non d\'eg\'en\'er\'ee 
dans le pinceau qui contient $3H$.
Sur un corps de nombres, ce r\'esultat lui permet d'\'etablir  le
 principe de Hasse pour les points rationnels
des intersections lisses de deux quadriques  dans $\P^7_{k}$ \cite[Thm. 1]{HB18}. 

On va maintenant retrouver ces r\'esultats.

Pour le r\'esultat local, on va   utiliser le th\'eor\`eme \ref{dansP4}.

 \medskip

\begin{thm}\label{droitesinglocale} 
Soit  $k$ un corps $p$-adique. Soit $X \subset \P^7_{k}$ une intersection compl\`ete lisse
de deux quadriques, donn\'ee par un syst\`eme $f=g=0$.
Supposons $X(k)$ non vide. Alors :

(a) Il existe sur $X$ un couple de droites distinctes, s\'ecantes, et globalement d\'efinies sur $k$.

(b)   Il existe une forme quadratique non d\'eg\'en\'er\'ee $\lambda f+ \mu  g$, avec $\lambda, \mu  \in k$, 
qui contient $3H$.
\end{thm}

\begin{proof}
Soit   $M \in X(k)$. Comme $k$ est $p$-adique, ou encore parce qu'une 
  intersection compl\`ete lisse $X$ de deux quadriques dans $\P^n_{k}$, $n \geq 4$
est $k$-unirationnelle d\`es qu'elle a un $k$-point,
 les $k$-points sont Zariski denses dans $X$. 
 D'apr\`es le th\'eor\`eme \ref{generalintegre} (f), 
 on peut choisir $M$
 de sorte que l'espace tangent  \`a $X$ en $M$
 d\'ecoupe sur $X \subset \P^7_{k}$ un c\^{o}ne de dimension relative 1 sur
 une intersection compl\`ete lisse  $Y$ de deux quadriques dans $\P^{4}$.

 D'apr\`es le th\'eor\`eme \ref{dansP4}, 
 la $k$-vari\'et\'e $Y \subset \P^4_{k}$
  a un point
dans une extension au plus quadratique $K/k$, et comme $Y$ est lisse, ses $K$-points sont
Zariski denses sur $Y_{K}$.  Si $Y(k) \neq \emptyset$, les 
$k$-points sont denses sur $Y$.
 On peut donc trouver  sur $X$ un couple de droites
distinctes, soit d\'efinies toutes deux sur $k$, soit d\'efinies sur une
extension quadratique $K$ et conjug\'ees sur ce corps, et se rencontrant en le $k$-point $M$.

Il existe au moins une forme quadratique $\lambda f + \mu g$ dans le pinceau,
d\'efinie sur $k$, qui s'annule sur le plan $\P^2_{k}$ engendr\'e par ces deux droites.

Si la forme quadratique $\lambda f + \mu g$  est de rang 8, alors elle contient $3H$.
Si la forme quadratique est de rang 7, la proposition \ref{coniqueplan}
montre qu'il existe une forme de rang 8 dans le pinceau qui contient
$3H$. On pourrait aussi utiliser 
 la proposition plus sp\'eciale \ref{dansP7} (vii).
\end{proof}

  \begin{rmk}
Partant de $P \in X(k)$ quelconque, la proposition \ref{para3CTSaSD} 
donne d\'ej\`a  un c\^{o}ne de sommet $P$ sur 
une intersection de deux quadriques $Y \subset \P^4_{k}$,
sans la pr\'ecision qu'on peut  trouver $Y$ lisse.  Sur le corps $p$-adique $k$,
 le th\'eor\`eme \ref{dansP4} assure l'existence d'un point de $Y$
 sur une extension quadratique de $k$.
  \end{rmk}

 \medskip

Voici maintenant une d\'emonstration alternative du th\'eor\`eme global
\'etabli par Heath-Brown \cite[Thm. 1]{HB18}.

\begin{thm}
Soient $k$ un corps de nombres et $X \subset \P^7_{k}$ une intersection compl\`ete
lisse de deux quadriques.  Le principe de Hasse vaut pour les points rationnels de $X$.
\end{thm}

\begin{proof} On part de $X$ avec $X(k_{v})\neq \emptyset$ pour toute place $v$.
On utilise   la proposition \ref{mordell} (le cas $k=\R$), 
la proposition \ref{dansP7} (v) (le cas de bonne r\'eduction, qu'il suffit d'ailleurs de conna\^{\i}tre 
pour les compl\'et\'es $k_{v}$ hors d'un ensemble fini de places de $k$),
et le th\'eor\`eme \ref{droitesinglocale}.
Par approximation faible, on trouve    $\lambda \in k$ avec $f+\lambda g$ de rang 8
qui contient $3H$ sur chaque compl\'et\'e de $k$, et qui donc
par le th\'eor\`eme \ref{Hasse} contient $3H$. Donc $X$
contient une conique, et on conclut
avec le th\'eor\`eme \ref{sal93}  (Salberger).
\end{proof}

Terminons cette section avec une remarque.
On aurait pu essayer d'\'etablir le th\'eor\`eme en \'etablissant l'existence
d'un couple de droites conjugu\'ees sur $X$. Dans cette direction, on
peut au moins \'etablir le r\'esultat suivant.

    \begin{prop}\label{Hassedroitequad}
  Soient $k$ un corps de nombres et  $X \subset \P^7_{k}$ lisse d\'efinie par $f=g=0$,
  avec des points dans tous les compl\'et\'es.
  Il existe une extension quadratique $K/k$ telle que
  $f+tg$ sur le corps $K(t)$ contienne $2H$
  sur les compl\'et\'es $K_{w}(t)$, autrement dit $F_{1}(X)(\A_{K})\neq \emptyset$.
    \end{prop}
\begin{proof}
  Pour presque toute place $v$ de $k$, la forme g\'en\'erique $f+tg$ contient
 $3H$ sur $k_{v}(t)$, et $X$ contient un plan $\P^2_{k_{v}}$.
  Pour toute place $v$ de $k$,  il existe une extension quadratique $K_{w}$ de $k_{v}$
  telle que  la forme g\'en\'erique $f+tg$ contienne $2H$  sur $K_{w}(t)$
  (car on a vu au th\'eor\`eme
  \ref{droitesinglocale}  qu'il existe une droite de $X$  d\'efinie sur
  une extension quadratique de $k_{v}$).
  Par approximation faible, on trouve une extension quadratique $K/k$
  telle que $f+tg$ sur $K(t)$ contienne $2H$  sur tous les $K_{w} (t)$
  pour $w$ place de $k$. 
  On peut choisir $K$ totalement imaginaire.
\end{proof}

  La proposition \ref{dansP7} (i) implique que la $k$-vari\'et\'e projective et lisse
  g\'eom\'etriquement rationnelle  $F_{1}(X)$ satisfait
  $\Br(k)=\Br(F_{1}(X))$.  
 On s'attend  donc \`a avoir le principe de Hasse pour l'existence de droites
  sur $X \subset \P^7$.  Avec les notations de la proposition
    \ref{dansP7}, $X$ qui a des points
  dans tous les compl\'et\'es de $k$ devrait selon la proposition \ref{Hassedroitequad}
contenir une droite sur une extension quadratique
  et donc d'apr\`es \cite{CTSaSD87} devrait avoir un $k$-point.
  Mais  prouver  le principe de Hasse pour $F_{1}(X)$ ne semble pas a priori plus simple
  que de prouver le principe de Hasse pour   $X$.

  \medskip

{\bf Remerciements.}   Les travaux r\'ecents de Creutz et Viray \cite{CV21} et de Iyer et Parimala \cite{IP22} m'ont amen\'e \`a revenir sur ce sujet.
Je remercie Aleksandr Kuznetsov d'avoir bien voulu \'etablir  et r\'ediger les
r\'esultats rassembl\'es dans l'appendice.  Le rapport critique de l'arbitre a permis une meilleure pr\'esentation
des r\'esultats arithm\'etiques de l'article.

 \appendix 
 
  \section{Appendix, by A. Kuznetsov}
   
  \def\kk{{k}}
  \def\Gr{{\rm G}}
  \def\bkk{ {\overline k} }
\def\cQ{   \mathcal{Q}     }
 \def\cV{   \mathcal{V}     }
 \def\cO{   \mathcal{O}     }
  \def\cE{   \mathcal{E}     }
   \def\cF{   \mathcal{F}     }
  \def\cU{   \mathcal{U}     }
 \def\hcV  {  \widehat{\cV}  }
 \def\hV{   \widehat{V}   }
 \def\hG {   \widehat{G}   }
\def\hU{    \widehat{U}   }
\def\OGr{{\rm OGr}}
  \def\hq {\hat{q}}
 \def\cS{   \mathcal{S}     }
\def\io{\iota}
\def\upbeta{\beta}
  \def\Cliff{{\mathcal{C}\!\ell}}
   \def\cB{   \mathcal{B}     }
   \def\cN{   \mathcal{N}     }
  \def\cHom{\mathcal{H}\mathit{om}}
 \def\rank{{\rm{rank}}}
 \def\Ker{{\rm Ker}}
 \def\eps{{\varepsilon}}
 \def\Hom{{\rm Hom}}
 \def\tD{   \widetilde{D}   }
\def\rd{\mathrm{d}}
\def\rs{\mathrm{s}}
\def\id{{\rm id}}
\def\Bl{{\rm Bl}}
  
We work over an arbitrary  field~$\kk$ of characteristic not equal to~$2$.
Let~$V$ be a vector space of dimension~$n + 1$  and let 
\begin{equation*}
X = Q_1 \cap Q_2 \subset \P(V) = \P^n
\end{equation*}
be a smooth complete intersection of two quadrics.
Over a separable closure of~$\kk$, the smoothness of~$X$ implies that the pencil generated by~$Q_1$ and~$Q_2$
contains exactly~$n + 1$ quadrics of corank~$1$, and all the other quadrics in the pencil are nondegenerate.
Moreover, the vertices of singular quadrics in the pencil do not lie on~$X$.

\subsection{Relative Hilbert schemes of planes}

Let~$G_r(X) \subset \Gr(r+1,V) \times \P^1$ be the relative Hilbert scheme of linear spaces $\P^r$
(linearly embedded into~$\P^n$)
contained in the quadrics of the pencil generated by~$Q_1$ and~$Q_2$.

\begin{prop}
\label{prop:gr}
Let $\kk$ be a field of characteristic not equal to~$2$.
If~\mbox{$n \ge 2r + 1$} the scheme~$G_r(X)$ is smooth and geometrically connected 
of expected dimension \mbox{$(r + 1)(n - \tfrac32r - 1) + 1$ over~$\kk$}.

Moreover, if~$n = 2r + 1$ the morphism~$G_r(X) \to \P^1$ factors as
\begin{equation*}
G_r(X) \to C \to \P^1,
\end{equation*}
where~$C$ is a smooth geometrically connected curve, 
$C \to \P^1$ is a double covering branched at the discriminant locus of the pencil,
which has length~$n + 1 = 2r + 2$ and whose~$\bkk$-points correspond to degenerate quadrics in the pencil,
and~$G_r(X) \to C$ is a smooth morphism with geometrically connected fibers.
\end{prop}

The smoothness of $G_r(X)$ is essentially proved in~\cite[Proposition~2.1]{Ku11}.
 
{\it Proof.}
To prove smoothness (and geometric connectedness)  of~$G_r(X)$
we may pass to the algebraic closure of the base field,
so we assume~$\kk$ algebraically closed.
First, the smoothness of~$X$ implies the smoothness of the total space~$\cQ$ of the relative quadric in the pencil
(because~$\cQ$ is nothing but the blowup of~$\P(V)$ at~$X$).
Then the first part of~\cite[Proposition~2.1]{Ku11} implies that the natural morphism
\begin{equation*}
T_P\P^1 \to S^2K_P^\vee
\end{equation*}
is surjective for any~$P \in \P^1$, where~$K_P \subset V$ 
is the kernel of the quadratic form corresponding to the point~$P$.
Finally, the argument of the second part of~\cite[Proposition~2.1]{Ku11} 
implies that~$G_r(X)$ is smooth of expected dimension.

Next, note that the fiber of~$G_r(X)$ over a point~$P \in \P^1$ is the subscheme of the Grassmannian~$\Gr(r+1,n+1)$
parameterizing subspaces~$U \subset V$ of dimension~$r + 1$
isotropic for the quadratic form~$q_P$ of the quadric~$\cQ_P$.
In particular, if~$P$ is not in the discriminant, so that the form~$q_P$ is nondegenerate, and~$n \ge 2r + 2$,
this is the homogeneous variety~$\OGr_{q_P}(r + 1, V)$ of the special orthogonal group of the form~$q_P$, 
hence it is smooth and connected,
and if~$n = 2r + 1$, this is a disjoint union of two smooth homogeneous varieties~$\OGr_{q_P}^\pm(r+1,2r+2)$.

On the other hand, if~$n = 2r + 1$ and~$P$ belongs to the discriminant, 
so that the form~$q_P$ has a $1$-dimensional kernel space~$K_P$,
then every $q_P$-isotropic subspace contains~$K_P$, 
and the map~$U \mapsto U/K_P$ induces an isomorphism 
between the fiber of~$G_r(X)$ over~$P$ with the subscheme of~$\Gr(r,n)$ 
parameterizing subspaces isotropic for the quadratic form~$\bar{q}_P$ induced by~$q_P$ on~$V/K_P$,
which is also a homogeneous variety, $\OGr_{\bar{q}_P}(r,V/K_P)$, hence smooth and connected.

To prove
connectedness of~$G_r(X)$,
first assume that~$n \ge 2r + 2$.
Then all fibers of~$G_r(X) \to \P^1$ are connected, so it follows that~$G_r(X)$ is connected.

Finally, assume~$n = 2r + 1$, and let again~$\kk$ be any field of characteristic not equal to~$2$.
Let
\begin{equation*}
G_r(X) \xrightarrow{\quad} C \xrightarrow{\ f\ } \P^1
\end{equation*}
be the Stein factorization for the map~$G_r(X) \to \P^1$.
Since~$G_r(X)$ is normal, $C$ is a normal curve.
Since the geometric general fiber of~$G_r(X)$ has two connected components,
$C$ is the normal closure of~$\P^1$ in the corresponding quadratic field extension of the field of rational functions,
hence~$C \to \P^1$ is a double covering ramified over the discriminant.
Since the discriminant is reduced and the characteristic is not~$2$,
the curve~$C$ is smooth over~$\kk$ and geometrically connected.

To prove that the morphism~$G_r(X) \to C$ is smooth with geometrically connected fibers
we may pass to the algebraic closure of the base field,
so from now on we assume~$\kk$ algebraically closed.
Let~$x_1,\dots,x_{n+1} \in C$ be the ramification points of~$f$.
Following the proof of~\cite[Proposition~2.7]{FK18} we consider the vector bundles
\begin{equation*}
\cV \coloneqq V \otimes \cO_{\P^1}
\qquad\text{and}\qquad 
\hcV \coloneqq \bigoplus_{i=1}^{n+1} \cO_C(x_i)
\end{equation*}
on~$\P^1$ and~$C$, respectively, endowed with the quadratic forms
\begin{equation*}
q \colon \cV \to \cV^\vee \otimes \cO_{\P^1}(1)
\qquad\text{and}\qquad 
\hq \colon \hcV \to \hcV^\vee \otimes f^*\cO_{\P^1}(1),
\end{equation*}
where the first 
map
is induced by the pencil of quadrics, and the second
map
 is defined as the direct sum of morphisms
\begin{equation*}
\cO_C(x_i) \cong \cO_C(-x_i) \otimes \cO_C(2x_i) \cong \cO_C(-x_i) \otimes f^*\cO_{\P^1}(1),
\end{equation*}
in particular the form~$\hq$ is everywhere non-degenerate.
It is easy to see that we have the following commutative diagram
\begin{equation*}
\xymatrix@C=5em{
f^*\cV \ar[r]^-{f^*q} \ar[d]_\io &
f^*\cV^\vee \otimes f^*\cO_{\P^1}(1)
\\
\hcV \ar[r]^-{\hq} &
\hcV^\vee \otimes f^*\cO_{\P^1}(1), \ar[u]_{\io^\vee}
}
\end{equation*}
where the left vertical arrow is defined as the composition
\begin{equation*}
\io \colon f^*\cV \cong \bigoplus_{i=1}^{n+1} \cO_C \hookrightarrow \bigoplus_{i=1}^{n+1} \cO_C(x_i) = \hcV
\end{equation*}
with the middle arrow being the direct sum of the embeddings~$\cO_C \hookrightarrow \cO_C(x_i)$,
and the right vertical arrow is its dual.
In particular, away from the points~$x_i$ the vertical arrows are isomorphisms, 
hence the quadratic forms agree.
On the other hand, over~$x_i$ the map~$\io$ factors as the composition
\begin{equation*}
V \xrightarrow{\ \pi_i\ } V/K_{P_i} \hookrightarrow \hV,
\end{equation*}
where~$\pi_i$ is the projection, and~$K_{P_i} \subset V$ is the $1$-dimensional kernel of the quadratic form corresponding to the point~$P_i$.

Now consider the relative isotropic Grassmannian~$\hG_r(X) \to C$ 
that parameterizes vector subspaces of dimension~$r + 1$ in the fibers of~$\hcV$ 
isotropic with respect to~$\hq$.
The argument of~\cite[Proposition~2.7]{FK18} shows that~$\hG_r(X)$ has two connected components
\begin{equation*}
\hG_r(X) = \hG_r^+(X) \sqcup \hG_r^-(X),
\end{equation*}
isomorphic to each other.
We claim that each of these components is isomorphic to~$G_r(X)$.
Indeed, consider the morphism
\begin{equation}
\label{eq:map-gplus-g}
\hG_r^+(X) \to G_r(X),
\qquad 
(x,\hU) \mapsto (f(x),\io^{-1}(\hU)).
\end{equation}
It is obvious that this map is well defined and is an isomorphism away from 
the point~$P_i$.
On the other hand, over~$P_i$ the map factors as
\begin{equation*}
\OGr^+_{\hq_{x_i}}(r+1,\hV) \to
\OGr_{\bar{q}_{P_i}}(r,V/K_{P_i}) \to
\OGr_{q_{P_i}}(r+1,V),
\end{equation*}
where~$\bar{q}_{P_i}$ is the (non-degenerate) quadratic form induced on~$V/K_{P_i}$ by the form~$q_{P_i}$,
the first map is given by~$\hU \mapsto \hU \cap (V/K_{P_i})$, 
and the second by~$\bar{U} \mapsto \pi_i^{-1}(\bar{U})$.
It is easy to see that both maps are isomorphisms, hence so is their composition,
and therefore the map~\eqref{eq:map-gplus-g} is an isomorphism.

It remains to note that the natural map~$\hG^+_r(X) \to C$ is smooth 
(because~$\hq$ is everywhere nondegenerate) with connected fibers,
hence~$\hG^+_r(X)$ is connected, hence~$G_r(X)$ is connected.
QED

\bigskip

We need to study the case~$r = 2$, $n = 5$ in more detail.
Let~$f \colon C \to \P^1$ be the double covering constructed in Proposition~\ref{prop:gr}
and let~$\Cliff_0$ be the sheaf of even parts of Clifford algebras over~$\P^1$ 
associated with the pencil of quadrics.
Recall from~\cite[\S3.5]{Ku08} that there is an Azumaya algebra~$\cB_0$ on~$C$ such that
\begin{equation*}
\Cliff_0 \cong f_*\cB_0.
\end{equation*}
We denote by~$\upbeta \in \Br(C)$ the Brauer class of~$\cB_0$.

\begin{prop}\label{prop:sb}
Let $\kk$ be a field of characteristic not equal to~$2$.
If~$n = 5$ the morphism~$G_2(X) \to C$ constructed in Proposition~\textup{\ref{prop:gr}}
is a $3$-dimensional Severi--Brauer variety of class~$\upbeta$.
\end{prop}

{\it Proof.}
Since~$\cB_0$ is an Azumaya algebra of rank~$2^{n-1} = 16$, 
there is a $\upbeta$-twisted locally free sheaf~$\cS$ on~$C$ such that~$\cB_0 \cong \cHom(\cS,\cS)$.
Then up to line bundle twists the bundle~$\hcV$ is isomorphic to both~$\wedge^2\cS$ and~$\wedge^2\cS^\vee$, 
the corresponding quadric bundle is isomorphic to the relative Grassmannian~$\Gr_C(2,\cS) \cong \Gr_C(2,\cS^\vee)$,
and there is a standard canonical isomorphism
\begin{equation*}
\hG_2(X) \cong \P_C(\cS) \sqcup \P_C(\cS^\vee).
\end{equation*}
Therefore, $G_{2}(X) \cong \P_C(\cS) \cong \P_C(\cS^\vee)$ by the argument of Proposition~\ref{prop:gr}.
QED

\subsection{Absolute Hilbert schemes of planes}

Let~$F_r(X) \subset \Gr(r+1, V)$ be the Hilbert scheme of linear spaces $\P^r$ (linearly embedded into~$\P^n$) contained in~$X$.   

\begin{lemang}
\label{lem:fr}
Let $\kk$ be a field of characteristic not equal to~$2$.
If~$n \le 2r + 1$ the scheme~$F_r(X)$ is empty, 
and if~$n \ge 2r + 2$ it is smooth over~$\kk$ and nonempty of expected dimension $(r+1)(n-2r-2)$.
\end{lemang}

{\it Proof.}
We may (and will) assume that~$\kk$ is algebraically closed.
Let~$U \subset V$ be the vector subspace of dimension~$r + 1$ corresponding to a point~$[U] \in F_r(X)$.
Let~$U^\perp \subset V^\vee$ be the kernel of the restriction map~$V^\vee \to U^\vee$.
If~$n \le 2r$ then $\dim(U^\perp) = n - r < r + 1 = \dim(U)$, 
and since each quadratic form in the pencil takes~$U$ to~$U^\perp \subset V^\vee$ (by the assumption~$[U] \in F_r(X)$),
it follows that every quadratic form is degenerate, hence~$X$ is singular.
On the other hand, if~$n = 2r + 1$ then the determinant of every quadratic form~$q$ in the pencil
is equal to the square of the determinant of~$q\vert_U \colon U \to U^\perp$,
hence the discriminant subscheme in~$\P^1$ is nonreduced, hence~$X$ is also singular.

To compute the dimension and prove the smoothness of~$F_r(X)$ let~$[U] \in F_r(X)$.
Since the expected dimension of~$F_r(X)$ is 
\begin{equation*}
\dim(\Gr(r+1,n+1)) - 2\,\rank(S^2\cU^\vee) = (r+1)(n - 2r - 2),
\end{equation*}
it is enough to show that~$\dim(H^0(\P(U),\cN_{\P(U)/X})) = (r+1)(n - 2r - 2)$. 
Consider the standard exact sequence
\begin{equation}
\label{eq:cn-pu-x}
0 \to \cN_{\P(U)/X} \to V/U \otimes \cO_{\P(U)}(1) \to \cO_{\P(U)}(2) \oplus \cO_{\P(U)}(2) \to 0
\end{equation}
(here the middle term is~$\cN_{\P(U)/\P(V)}$ and the right term is~$\cN_{X/\P(V)}\vert_{\P(U)}$).
Its cohomology exact sequence shows that it is enough to check that~\mbox{$H^1(\P(U),\cN_{\P(U)/X})) = 0$};
to do this we use a trick.
Since~$X$ is smooth, $\cN_{\P(U)/X}$ is a vector bundle of rank~$n - r - 2$, 
and~\eqref{eq:cn-pu-x} shows its determinant is~$\cO_{\P(U)}(n - r - 4)$, hence
\begin{equation*}
\cN_{\P(U)/X} \cong \wedge^{n - r - 3}\cN^\vee_{\P(U)/X} \otimes \cO_{\P(U)}(n - r - 4).
\end{equation*}
Dualizing~\eqref{eq:cn-pu-x}, taking its wedge power, and twisting,
we obtain a long exact sequence
\begin{multline*}
0 \to \cO_{\P(U)}(r + 2 - n)^{\oplus (n - r - 2)} \to \dots \to 
\wedge^{n - 2r - 1}U^\perp \otimes \cO_{\P(U)}(1-r)^{\oplus (r-1)} \to \\ \dots \to
\wedge^{n - r - 4}U^\perp \otimes \cO_{\P(U)}(-2)^{\oplus 2} \to
\wedge^{n - r - 3}U^\perp \otimes \cO_{\P(U)}(-1) \to
\cN_{\P(U)/X} \to 0,
\end{multline*}
and its hypercohomology spectral sequence proves~$H^{> 0}(\P(U),\cN_{\P(U)/X}) = 0$.

It remains to show that~$F_r(X)$ is nonempty for~$n \ge 2r + 2$.
By Bertini theorem a general linear section~$X' = X \cap \P^{2r + 2}$ is smooth
and obviously~$F_r(X') \subset F_r(X)$, so we may assume~$n = 2r + 2$.
In this case the scheme~$F_r(X)$ is finite and by intersection theory its length is equal to
\begin{equation*}
\mathrm{c}_{(r+1)(r+2)}(S^2\cU^\vee \oplus S^2\cU^\vee) =
\mathrm{c}_{\binom{r+2}{2}}(S^2\cU^\vee)^2 = 2^{2r + 2},
\end{equation*}
in particular, $F_r(X)$ is nonempty.
QED

\begin{lemang}
\label{lem:fr-irred}
Let $\kk$ be a field of characteristic not equal to~$2$.
If~$n \ge 2r + 3$ the scheme~$F_r(X)$ is geometrically connected.
\end{lemang}

{\it Proof.}
We may assume that the field~$\kk$ is algebraically closed and prove connectedness of~$F_r(X)$.
By Lemma~\ref{lem:fr} the scheme~$F_r(X)$ is not empty.
Let~$U_0 \subset V$ be a vector subspace of dimension~$r + 1$ such that~$[U_0] \in F_r(X)$, i.e., $\P(U_0) \subset X$.

First, we prove that there is a hyperplane~$H \subset \P(V)$ containing~$\P(U_0)$ such that~$X \cap H$ is smooth.
We use a Bertini argument.
Assume~$X \cap H$ is not smooth at a point~$x \in X \setminus \P(U_0)$;
then~$H$ is equal to the embedded tangent space~$\mathrm{T}_x(Q)$ at~$x$ to some quadric~$Q$ in the pencil.
This means that~$U_0$ is orthogonal to~$x$ with respect to~$Q$,
i.e., the linear span~$\langle \P(U_0), x \rangle$ is contained in~$Q$.
Therefore, the map 
\begin{equation*}
\{ (x,Q) \in (X \setminus \P(U_0)) \times \P^1 \mid \langle \P(U_0), x \rangle \subset Q \} \to \P(V^\vee),
\qquad 
(x,Q) \mapsto \mathrm{T}_x(Q)
\end{equation*}
is surjective onto the variety of all~$H$ containing~$\P(U_0)$ such that~$X \cap H$ is singular away from~$\P(U_0)$.
But the source of this map is fibered over~$\P^1$ with fiber isomorphic to the quadric in~$\P(U_0^\perp/U_0)$ induced by~$Q$,
hence its dimension is equal to
\begin{equation*}
\dim(\P^1) + \dim(U_0^\perp/U_0) - 2 =
1 + (n + 1 - 2r - 2) - 2 = n - 2r - 2.
\end{equation*}
Since this is less than~$\dim(\P(U_0^\perp)) = n - r - 1$, the dimension of the variety of hyperplanes containing~$\P(U_0)$,
we conclude that for a general such hyperplane~$H$ the intersection~$X \cap H$ is smooth away from~$\P(U_0)$.

On the other hand, we note that~$X \cap H$ is smooth along~$\P(U_0)$ if and only if 
the induced section of the twisted conormal bundle~$\cN^\vee_{\P(U_0)/X}(1)$ has no zeroes.
But the sequence~\eqref{eq:cn-pu-x} implies that this vector bundle is globally generated
by the space~$U_0^\perp$ of hyperplanes~$H$ containing~$\P(U_0)$, 
hence a general such section has no zeroes as soon as the rank of the bundle is greater than the dimension of the base, i.e.,
\begin{equation*}
n - 2 - r > r.
\end{equation*}
Thus, if~$n \ge 2r + 3$, for a general hyperplane~$H \subset \P(V)$ containing~$\P(U_0)$
the intersection~$X \cap H$ is smooth along~$\P(U_0)$.

A combination of the above two observations shows that~$X \cap H$ 
is smooth for a general~$H \subset \P(V)$ containing~$\P(U_0)$. 
Iterating this argument we deduce that, if~\mbox{$n \ge 2r + 3$}, then
for a general subspace~$V_0 \subset V$ of dimension~$2r + 4$ containing~$U_0$
the intersection~$X \cap \P(V_0)$ is smooth.

Now consider the variety
\begin{equation*}
\tilde{F}_r(X) \coloneqq 
\{ (U_0,V_0) \in F_r(X) \times \Gr(2r + 4, V) \mid \text{$X \cap \P(V_0)$ is smooth} \}.
\end{equation*}
The second projection~$\tilde{F}_r(X) \to \Gr(2r + 4, V)$, $(U_0,V_0) \mapsto V_0$, 
is then a fibration over a dense open subset of the Grassmannian, 
and by~\cite[Theorem~4.8]{Reid72} its fiber over a point~$[V_0]$
is the Jacobian of the smooth hyperelliptic curve associated to~$X \cap \P(V_0)$.
In particular, all fibers are connected, hence~$\tilde{F}_r(X)$ is connected.
On the other hand, the first projection~$\tilde{F}_r(X) \to F_r(X)$, $(U_0,V_0) \mapsto U_0$, is surjective,
hence~$F_r(X)$ is connected as well.
QED

\subsection{Springer resolutions}

In this section we work over any field (even characteristic~$2$ is allowed)
and we prove the following general result.

Let~$\varphi \colon \cE \to \cF^\vee$ be a morphism of vector bundles over a scheme~$Z$ 
of ranks~$r_\cE$ and~$r_\cF$, respectively,
and let~$\varphi^\vee \colon \cF \to \cE^\vee$ be its dual morphism.
Sometimes it is convenient to (uniformly) consider~$\varphi$ and~$\varphi^\vee$ as a linear map~$\cE \otimes \cF \to \cO_Z$.
Let 
\begin{equation*}
D_r = 
\{\wedge^{r+1}\varphi = 0 \}
 = \{\wedge^{r+1}\varphi^\vee = 0 \} \subset Z
\end{equation*}
be the rank~$r$ degeneracy locus of the morphism~$\varphi$ (or, equivalently, of its dual~$\varphi^\vee$), 
where~$\wedge^{r+1}\varphi \colon \wedge^{r+1}\cE \to \wedge^{i+1}\cF^\vee$ 
and~$\wedge^{r+1}\varphi^\vee \colon \wedge^{r+1}\cF \to \wedge^{i+1}\cE^\vee$ 
are the wedge powers of~$\varphi$ and~$\varphi^\vee$.
Note that~$D_r \subset Z$ is a closed subscheme and~$D_{r-1} \subset D_r$ for each~$r$.

For each closed point~$z \in D_r$ let~$\kk(z)$ be the residue field of~$z$,
let~$\varphi_z \colon \cE_z \to \cF_z^\vee$ be the fiber of~$\varphi$ at~$z$, and let
\begin{equation*}
K_{\cE,z} \coloneqq \Ker(\varphi_z) \subset \cE_z
\qquad\text{and}\qquad
K_{\cF,z} \coloneqq \Ker(\varphi_z^\vee) \subset \cF_z
\end{equation*}
be the kernels of~$\varphi_z$ and~$\varphi_z^\vee$  respectively.
For any tangent vector to~$Z$ at~$z$, i.e., for an embedding~$\tau \colon \Spec(\kk(z)[\eps]/\eps^2) \to Z$
that takes the closed point to~$z$, consider the pullback~\mbox{$\tau^*\varphi \colon \tau^*\cE \to \tau^*\cF^\vee$}.
Trivializing the bundles~$\tau^*\cE$ and~$\tau^*\cF$, we can write~$\tau^*\varphi$ as~$\varphi_z + \eps \varphi_\tau$,
where~$\varphi_\tau \colon \cE_z \to \cF_z^\vee$ is a $\kk(z)$-linear map, 
which is well defined (i.e., does not depend on the choice of the trivializations) modulo~$\varphi_z$.
Again, sometimes it is more convenient to consider~$\varphi_\tau$ as a linear map~$\cE_z \otimes \cF_z \to \kk(z)$.
In particular, the maps~$\varphi_\tau\vert_{K_{\cE,z}} \colon K_{\cE,z} \to \cF_z^\vee$ 
and~$\varphi_\tau^\vee\vert_{K_{\cF,z}} \colon K_{\cF,z} \to \cE_z^\vee$ are well defined.

We will say that~$\varphi$ is {\sf $s$-regular at~$z$} 
if for any $s$-dimensional subspace~$U \subset K_{\cE,z}$ 
the morphism
\begin{equation}
\label{eq:tzz-hom}
T_z Z \to \Hom(K_{\cF,z},\cE_z^\vee)
\to \Hom(K_{\cF,z},U^\vee) = U^\vee \otimes K_{\cF,z}^\vee, 
\quad 
\tau \mapsto \varphi_\tau\vert_{U \otimes K_{\cF,z}}
\end{equation}
is surjective.
We define $s$-regularity of~$\varphi^\vee$ analogously.

\begin{remark}
\label{rem:regularity}
A morphism~$\varphi$ is $(r_\cE - r)$-regular at a geometric point~$z \in D_r \setminus D_{r-1}$
if and only if its dual~$\varphi^\vee$ is~$(r_\cF - r)$-regular at the point~$z$.
Indeed, both properties are equivalent to the surjectivity of the natural morphism~$T_z Z \to K_{\cE,z}^\vee \otimes K_{\cF,z}^\vee$.
\end{remark}

Consider the relative Grassmannians 
\begin{equation*}
p_\cE \colon \Gr_Z(r_\cE - r,\cE) \to Z 
\qquad\text{and}\qquad 
p_\cF \colon \Gr_Z(r_\cF - r,\cF) \to Z, 
\end{equation*}
their tautological subbundles~$\cU_\cE \subset p_\cE^*\cE$ and~$\cU_\cF \subset p_\cF^*\cF$, and the subschemes
\begin{equation*}
\tD_{r,\cE} \subset \Gr_Z(r_\cE - r,\cE)
\qquad\text{and}\qquad 
\tD_{r,\cF} \subset \Gr_Z(r_\cF - r,\cF)
\end{equation*}
defined as the zero loci 
of the composed maps 
\begin{equation*}
\cU_\cE \hookrightarrow p_\cE^*\cE \xrightarrow{\ p_\cE^*\varphi\ } p_\cE^*\cF^\vee
\qquad\text{and}\qquad 
\cU_\cF \hookrightarrow p_\cF^*\cF \xrightarrow{\ p_\cF^*\varphi^\vee\ } p_\cF^*\cE^\vee.
\end{equation*}
We abusively call these schemes the {\sf Springer resolutions} of~$D_r$, even if they are not smooth.

\begin{prop}
\label{prop:spr}
Let $\kk$ be a field.
Let~$Z$ be a smooth scheme and let~$\varphi \colon \cE \to \cF^\vee$ be a morphism of vector bundles over~$Z$.
The Springer resolution $$\tD_{r,\cE} \subset \Gr_Z(r_\cE - r,\cE)$$ is smooth over~$\kk$ of expected codimension~$(r_\cE - r)r_\cF$
if and only if the morphism~$\varphi$ is $(r_\cE - r)$-regular at every geometric point of~$D_{r}$.

Moreover, the projection~$p_\cE$ induces a proper surjective morphism~$\tD_{r,\cE} \to D_r$ with geometrically connected fibers,
which is an isomorphism over~\mbox{$D_r \setminus D_{r-1}$};
in particular, if~$D_r \setminus D_{r-1}$ is dense in~$D_r$, this morphism is birational.
\end{prop}

{\it Proof.}
We may (and will) assume that~$\kk$ is algebraically closed.

Let~$(z,U)$ be a point of~$\tD_{r,\cE}$, i.e., $U \subset K_{\cE,z}$ is a subspace of dimension~$r_\cE - r$.
Since the morphism~$\varphi_z \colon \cE_z \to \cF_z^\vee$ vanishes on~$U$, its rank is at most~$r$, hence~$z \in D_r$.
This proves that the morphism~$p_\cE$ factors through~$D_r$.

Furthermore, at the point~$(z,U)$ the differential~$\rd\rs_\cE$ of the section 
\begin{equation*}
\rs_\cE \in \Gamma(\Gr_Z(r_\cE - r, \cE), \cU_\cE^\vee \otimes p_\cE^*\cF^\vee)
\end{equation*}
defining the subscheme~$\tD_{r,\cE}$ is a morphism~$T_{(z,U)}\Gr_Z(r_\cE - r, \cE) \to U^\vee \otimes \cF_z^\vee$.
Consider the natural exact sequence
\begin{equation*}
0 \to U^\vee \otimes (\cE_z/U) \to T_{(z,U)}\Gr_Z(r_\cE - r, \cE) \to T_zZ \to 0
\end{equation*}
(here the first term is the fiber of the relative tangent bundle of~$\Gr_Z(r_\cE - r, \cE)$ over~$Z$).
It is easy to see that the restriction of~$\rd\rs_\cE$ to~$U^\vee \otimes (\cE_z/U)$ is the map 
\begin{equation*}
U^\vee \otimes (\cE_z/U) \to U^\vee \otimes \cF_z^\vee
\end{equation*}
induced by the map~$\id_{U^\vee} \otimes \varphi_z \colon U^\vee \otimes \cE_z \to U^\vee \otimes \cF_z^\vee$
(recall that~$U \subset K_{\cE,z}$), 
hence its cokernel is~$U^\vee \otimes K_{\cF,z}^\vee$.
Furthermore, the morphism
\begin{equation*}
T_zZ \to U^\vee \otimes K_{\cF,z}^\vee 
\end{equation*}
induced by~$\rd\rs_\cE$ coincides with the map~\eqref{eq:tzz-hom},
hence this morphism is surjective if and only if the morphism~$\varphi$ is $(r_\cE - r)$-regular at~$z$.
This proves the first part of the proposition.

To prove the second part, note that the morphism~$p_\cE \colon \tD_{r,\cE} \to D_r$ 
is proper by construction and surjective by definition.
Moreover, it is \'etale over~$D_r \setminus D_{r-1}$, because the restriction of~$\rd\rs_\cE$ 
to the relative tangent space
\begin{equation*}
U^\vee \otimes (\cE_z/U) = K_{\cE,z}^\vee \otimes (\cE_z/K_{\cE,z})
\end{equation*}
is injective,
and it is injective over~$D_r \setminus D_{r-1}$
because for any point~$z \in D_r \setminus D_{r-1}$ the space~$K_{\cE,z}$ has dimension exactly~$r_\cE - r$,
hence the only subspace of dimension~$r_\cE - r$ that it contains is the space~$U = K_{\cE,z}$.
Therefore, the morphism~$p_\cE \colon \tD_{r,\cE} \to D_r$ is an isomorphism over~$D_r \setminus D_{r-1}$.
The last part is obvious.
QED

\subsection{Hilbert schemes of quadrics in~$X$}

Let~$S_r(X)$ be the Hilbert scheme of quadrics of dimension~$r-1$ inside~$X \subset \P^n$.

\begin{prop}
\label{prop:sr}
Let $\kk$ be a field of characteristic not equal to~$2$. 
If~$n \geq 2r+1$ the Hilbert scheme~$S_r(X)$ is a smooth geometrically connected scheme 
of expected dimension~$(r + 1)(n - \tfrac32r - 1) + 1$ over~$\kk$.
Moreover, $S_r(X)$ is birational to~$G_r(X)$, and if~$F_r(X) = \varnothing$ then~$S_r(X) \cong G_r(X)$.
Finally, the subscheme~$S_r^\circ(X) \subset S_r(X)$ parameterizing smooth quadrics is open and dense in~$S_r(X)$.
\end{prop}
 
{\it Proof.}
Consider the Grassmannian~$Z \coloneqq \Gr(r+1,V)$, let~$\cU \subset V \otimes \cO$ be its tautological bundle, and consider the morphism
\begin{equation*}
\cE \coloneqq \cO \oplus \cO \xrightarrow{\ \varphi\ } S^2\cU^\vee \eqqcolon \cF^\vee
\end{equation*}
induced by the pencil of quadrics.
Note that the zero locus~$D_0$ of~$\varphi$ is precisely the Hilbert scheme~$F_r(X)$.
Let~$D = D_1$ be the degeneracy locus of~$\varphi$, so that we have the inclusions
\begin{equation*}
F_r(X) = D_0 \subset D_1 \subset \Gr(r+1,V).
\end{equation*}
Consider the Springer resolutions of~$D_1$:
\begin{equation*}
\tD_{1,\cE} \subset \P^1 \times \Gr(r+1,V)
\quad\text{and}\quad
\tD_{1,\cF} \subset \Gr_{\Gr(r+1,V)}(r_\cF-1,S^2\cU) \cong \P_{\Gr(r+1,V)}(S^2\cU^\vee),
\end{equation*}

First, note that by definition~$\tD_{1,\cE}$ 
is the zero locus of the global section of the vector bundle~$\cO(1) \boxtimes S^2\cU^\vee$, 
induced by~$\varphi$.
Thus,
\begin{equation}
\label{eq:td1-gr}
\tD_{1,\cE} \cong G_r(X).
\end{equation}
On the other hand, note that~$\P_{\Gr(r+1,V)}(S^2\cU^\vee)$ is the Hilbert scheme of $(r-1)$-dimensional quadrics in~$\P(V)$,
and its subscheme~$\tD_{1,\cF}$ parameterizes those quadrics that lie on~$X$, hence
\begin{equation}
\label{eq:td1-sr}
\tD_{1,\cF} \cong S_r(X).
\end{equation}

Now we check that both~$\varphi$ and~$\varphi^\vee$ satisfy the assumptions of Proposition~\ref{prop:spr} for~$r = 1$.
Note that
\begin{equation*}
r_\cE - r = 2 - 1 = 1
\qquad\text{and}\qquad 
r_\cF - r = \tbinom{r+1}{2} - 1 = \tfrac{r^2+3r}2.
\end{equation*}
By Proposition~\ref{prop:gr} and~\eqref{eq:td1-gr} the scheme~$\tD_{1,\cE}$ is smooth, 
hence by Proposition~\ref{prop:spr} the morphism~$\varphi$ is~$1$-regular at every geometric point of~$D_1 \setminus D_0$,
and by Remark~\ref{rem:regularity} the dual morphism~$\varphi^\vee$
is~$((r^2+3r)/2)$-regular at every geometric point of~$D_1 \setminus D_0$.
On the other hand, Lemma~\ref{lem:fr} implies that the natural morphism
\begin{equation*}
T_z Z = U^\vee \otimes (V/U) \to S^2U^\vee \oplus S^2U^\vee = \cE_z^\vee \otimes \cF_z^\vee
\end{equation*}
is surjective at every geometric point~$z = [U]$ of~$D_0 = F_r(X)$, hence the morphism~$\varphi^\vee$ 
is also~$((r^2+3r)/2)$-regular at every point of~$D_0$.
Applying Proposition~\ref{prop:spr} again we conclude that~$S_r(X) \cong \tD_{1,\cF}$ is smooth.

Since~$G_r(X)$ is smooth and geometrically connected by Proposition~\ref{prop:gr}, it is irreducible.
It is easy to see that the fibers of~$p_\cE \colon G_r(X) \to D_r$ over~$D_0 = F_r(X)$ are isomorphic to~$\P^1$, hence
\begin{multline*}
\dim(p_\cE^{-1}(D_0)) = \dim(F_r(X)) + 1 = (r+1)(n-2r-2) + 1 \\ < (r + 1)(n - \tfrac32r - 1) + 1 = \dim(G_r(X)),
\end{multline*}
hence~$\tD_{1,\cE} \setminus p_\cE^{-1}(D_0)$ is dense in~$\tD_{1,\cE}$, and hence~$D_1 \setminus D_0$ is dense in~$D_1$.
Therefore, the morphisms~$G_r(X) \to D_1$ and~$S_r(X) \to D_1$ are both birational,
hence~$S_r(X)$ is birational to~$G_r(X)$, and even isomorphic if~$D_0 = F_r(X) = \varnothing$.
It also follows that 
\begin{equation*}
\dim(S_r(X)) = \dim(G_r(X)) = (r + 1)(n - \tfrac32r - 1) + 1.
\end{equation*}
Moreover, the above argument shows that~$D_1 \setminus D_0$ is geometrically irreducible 
and its (isomorphic) preimage is dense in~$S_r(X)$, hence~$S_r(X)$ is geometrically connected.

Finally, we consider the subscheme~$S_r^\circ(X) \subset S_r(X)$ parameterizing smooth quadrics.
Since smoothness is an open condition, $S_r^\circ(X)$ is open in~$S_r(X)$, so we only need to check it is dense,
and since~$S_r(X)$ is irreducible, it is enough to check that~$S_r^\circ(X) \ne \varnothing$.
For this we may (and will) assume that~$\kk$ is algebraically closed.

Let~$Q \subset \P^n$ be any smooth quadric containing~$X$.
Let~$F_r(Q)$ be the Hilbert scheme of linear spaces~$\P^r$ (linearly embedded into~$\P^n$) contained in~$Q$,
and let~$L_r(Q) \subset Q \times F_r(Q)$ be the universal linear space.
Note that~$F_r(Q)$ and~$L_r(Q)$ are homogeneous spaces of the orthogonal group associated with~$Q$;
in particular, the projection~$L_r(Q) \to Q$ is a smooth morphism.
Therefore, $L_r(Q) \times_Q X$ is smooth over~$\kk$.
On the other hand, the projection
\begin{equation*}
L_r(Q) \times_Q X \to F_r(Q) 
\end{equation*}
is a quadric fibration.
If it is not flat then~$F_r(X) \ne \varnothing$, hence~$X$ contains a~$\P^r$, and a fortiori a smooth quadric of dimension~$r-1$.
On the other hand, if it flat then Lemma~\ref{lem:gen-sm} below implies that there is a $\kk$-point~$[\P^r]$ of~$F_r(Q)$
such that the fiber over it is a smooth quadric of dimension~$r - 1$.
It remains to note that this fiber is the intersection~$X \cap \P^r$, hence a smooth quadric inside~$X$.
QED

\begin{remark}
Let~$SF_r(X) \subset S_r(X)$ be the subscheme of quadrics contained in~$X$ together with their linear span.
Then one can check that there is a diagram of birational maps
\begin{equation*}
\xymatrix{
&
\Bl_{\P^1 \times F_r(X)}(G_r(X)) \ar[dl] \ar@{=}[r] &
\Bl_{SF_r(X)}(S_r(X)) \ar[dr]
\\
G_r(X) &&&
S_r(X),
}
\end{equation*}
where the diagonal arrows are blowups,
and that the induced birational map
\begin{equation*}
S_r(X) \dashrightarrow G_r(X)
\end{equation*}
is a standard flip.
\end{remark}

In the proof of Proposition~\ref{prop:sr} we used the following lemma.

\begin{lemang}
\label{lem:gen-sm}
Let $\kk$ be a field of characteristic not equal to~$2$. 
Let~$f \colon \cQ \to Z$ be a flat quadric fibration.
If~$\cQ$ and~$Z$ are smooth over~$\kk$ there is a dense open subset~$Z_0 \subset Z$
such that the morphism~$\cQ_{Z_0} \to Z_0$ is smooth.
\end{lemang}

{\it Proof.}
We will argue by induction on relative dimension of~$\cQ$ over~$Z$.

If~$\dim(\cQ/Z) = 0$, the map~$f$ is a double covering, 
and since characteristic of~$\kk$ is not equal to~$2$, 
we can locally represent~$\cQ$ as a hypersurface~$\{y^2 = \phi(z)\} \subset \mathbb{A}^1 \times Z$,
where~$\phi(z)$ is not a zero divisor.
Therefore, the subset~$Z_0 = \{ \phi(z) \ne 0 \} \subset Z$ is the required dense open subset.

Now assume~$\dim(\cQ/Z) > 0$.
The question is local over~$Z$, so we may assume that~$\cQ \subset \P(V) \times Z$ 
is given by a quadratic equation in homogeneous coordinates of~$\P(V)$ with coefficients functions on~$Z$.
Moreover, shrinking~$Z$ if necessary we can choose a point~$P_0 \in \P(V)$ such that~$(P_0 \times Z) \cap \cQ = \varnothing$.
Consider the linear projection~$\P(V) \dashrightarrow \P(V')$ with center~$P_0$, 
and the induced morphism
\begin{equation*}
\pi \colon \cQ \to \P(V') \times Z.
\end{equation*}
The choice of the point~$P_0$ ensures that~$\pi$ is a double covering, 
and its branch divisor~\mbox{$\cQ' \subset \P(V') \times Z$} 
is a flat quadric fibration over~$Z$.
The smoothness of~$\cQ$ over~$\kk$ implies the smoothness of~$\cQ'$ over~$\kk$.
But~$\dim(\cQ'/Z) = \dim(\cQ/Z) - 1$, hence by induction hypothesis there is a dense open subset~$Z_0 \subset Z$
such that~$\cQ'_{Z_0} \to Z_0$ is smooth.
It is clear that~$\cQ_{Z_0}$ is the double covering of~$\P(V') \times Z_0$ ramified over~$\cQ'_{Z_0}$,
hence it is also smooth over~$Z_0$.
QED


\begin{thebibliography}{CTSS83}



	
\bibitem[AC17]{AC17} C. Araujo et C. Casagrande, On the Fano variety of linear spaces contained in two odd-dimensional quadrics,
Geometry \& Topology {\bf 21} (2017) 3009-3041.

\bibitem[ABB14]{ABB14}
A.  Auel,  M. Bernardara, M  Bolognesi,  Fibrations in complete intersections of quadrics, Clifford algebras, derived categories, and rationality problems. J. Math. Pures Appl. (9) 102 (2014), no. 1, 249--291.
 
\bibitem[CT88]{CT88}  J.-L. Colliot-Th\'el\`ene, \emph{Surfaces rationnelles fibr\'ees en coniques de degr\'e 4},  
  S\'eminaire de th\'eorie des nombres de Paris 88-89,  Progr. Math., t. 91 (1990), 43--55.  


\bibitem[CTCS80]{CTCS80}  J.-L. Colliot-Th\'el\`ene,  D. Coray et  J.-J.  Sansuc,  
Descente et principe de Hasse pour certaines vari\'et\'es rationnelles.   J. reine    angew. Math. {\bf 320} (1980), 150--191.
 
  \bibitem[CTSa80]{CTSa80} J.-L. Colliot-Th\'el\`ene et J.-J. Sansuc, La descente sur les vari\'et\'es rationnelles, in {\it Journ\'ees de g\'eom\'etrie alg\'ebrique d'Angers} (Juillet 1979), \'ed. A. Beauville, Sijthoff \& Noordhoff (1980), 223--237.


\bibitem[CTSaSD87]{CTSaSD87} 
 J.-L. Colliot-Th\'el\`ene,  J.-J.  Sansuc et  P. Swinnerton-Dyer,
  \emph{Intersections of two quadrics and Ch\^atelet surfaces I},    J. reine angew. Math. {\bf 373} (1987) 37-107.
 \emph{Intersections of two quadrics and Ch\^atelet surfaces II},    J. reine angew. Math. {\bf 374}(1987), 72-168.
 
\bibitem[CTSk93]{CTSk93}  J.-L. Colliot-Th\'el\`ene et A. N. Skorobogatov, 
Groupes de Chow des z\'ero-cycles des fibr\'es en quadriques, Journal of K-theory {\bf 7} (1993) 477--500.
 
\bibitem[CTSk21]{CTSk21} 
J.-L. Colliot-Th\'el\`ene et A. N. Skorobogatov, The Brauer--Grothendieck group, 
 {\it  The Brauer--Grothendieck group}, 
Ergebnisse der Mathematik und ihrer Grenzgebiete. 3. Folge, {\bf  71}. Springer, Cham,  2021.

\bibitem[CoTs88]{CoTs88} D. F. Coray et M. A. Tsfasman, Arithmetic on singular Del Pezzo surfaces,
Proc. London Math. Soc. (3) {\bf 57} (1988) 25--87.

\bibitem[CV21]{CV21} B. Creutz and B. Viray, Quadratic points on intersections of two quadrics,
 arXiv:2106.08560v4  [math.NT], to appear in Algebra \& Number Theory.
 
   
 \bibitem[DM98]{DM98} O. Debarre et  L. Manivel, Sur la vari\'et\'e des espaces lin\'eaires contenus dans une intersection compl\`ete,
 Math. Annalen {\bf 312} (1998) 549--574.
 
 \bibitem[FK18]{FK18}
A. Fonarev et A. Kuznetsov, 
Derived categories of curves as components of Fano manifolds. 
J. Lond. Math. Soc. (2) 97 (2018), no. 1, 24--46.
 
 \bibitem[Ha94]{Ha94} D. Harari, M\'ethode des fibrations et obstruction de Manin. Duke Math. J. 75 (1994), no. 1, 221--260. 


 
\bibitem[HWW21]{HWW21} Y. Harpaz, D. Wei et O. Wittenberg,  Rational points on fibrations with few non-split fibres,
pr\'epublication, 2021.

\bibitem[H24]{H24} H. Hasse,    Darstellbarkeit von Zahlen durch quadratische Formen in einem beliebigen algebraischen Zahlk\"{o}rper.   J. reine angew. Math. {\bf 153} (1924), 113--130. 
 


\bibitem[HT21]{HT21}  B. Hassett et  Yu. Tschinkel,
Rationality of complete intersections of two quadrics over nonclosed fields. With an appendix by Jean-Louis Colliot-Th\'el\`ene.  L'Enseignement math\'ematique {\bf  67} (2021), no. 1-2, 1-44.

  
 

 \bibitem[HB18]{HB18} R. Heath-Brown,  
 Zeros of pairs of quadratic forms, J. reine angew. Math. {\bf 739} (2018), 41--80. 
 
 	
\bibitem[IP22]{IP22}	 J. Iyer et R. Parimala, Period-index problem for hyperelliptic curves, tapuscrit, Janvier 2022.

\bibitem[Ka08]{Ka08}
B. Kahn,  Formes quadratiques sur un corps,  Cours sp\'ecialis\'es {\bf 15}, Soci\'et\'e math\'ematique de France 2008.

 
\bibitem[Ku08]{Ku08}
A. Kuznetsov,  
Derived categories of quadric fibrations and intersections of quadrics. 
Adv. Math. 218 (2008), no. 5, 1340-1369.
 
\bibitem[Ku11]{Ku11}
A. Kuznetsov, 
Scheme of lines on a family of 2-dimensional quadrics: geometry and derived category. 
Math. Z. 276 (2014), no. 3--4, 655--672.


\bibitem[Lam73]{Lam73} T. -Y. Lam, The Algebraic Theory of Quadratic Forms, Benjamin/Cummings, 1973.

\bibitem[Lam05]{Lam05} T. -Y. Lam, Introduction to Quadratic Forms over Fields, Graduate Studies in Mathematics {\bf 67}, Amer. Math. Soc. (2005).
 
\bibitem[Leep]{Leep} D. Leep,  The Amer--Brumer theorem over arbitrary fields, pr\'epublication.
 
 \bibitem[Li68]{Li68} S. Lichtenbaum,  The Period-Index Problem for Elliptic  Curves, Amer. J. Math. {\bf 90}, no. 4 (1968) 1209--1223.
 
 
 \bibitem[Li69]{Li69}  S. Lichtenbaum, Duality Theorems for Curves over $p$-adic Fields, Invent. math. {\bf 7} (1969) 120--136.
 
 
\bibitem[N75]{N75}  P.E. Newstead,  Rationality of moduli spaces of stable bundles, Math. Ann.{\bf  215} (1975), 251--268

\bibitem[Reid72]{Reid72} M. Reid. The complete intersection of two or more quadrics, Thesis, Trinity College, Cambridge,  June 1972.
 
 \bibitem[Sal88]{Sal88}  P. Salberger, Zero-cycles on rational surfaces over number fields, Invent. math. 91 (1988) 505--524.
 
 \bibitem[Sal89]{Sal89} P. Salberger, Some new Hasse principles for conic bundle surfaces, in 
 S\'eminaire de Th\'eorie des Nombres, Paris 1987--1988,  (1989) 283-305.
 
 \bibitem[Sal93]{Sal93} P. Salberger, On the intersection of two quadrics containing a conic, preprint (1993)
 	arXiv:2305.02289
	 
 \bibitem[SalSk91]{SalSk91} P. Salberger et A.N. Skorobogatov,  Weak approximation for surfaces defined by two quadratic forms, Duke Math. J. {\bf 63} no. 2 (1991) 517--536.

 

\bibitem[ZT17]{ZT17}  Zhiyu Tian, Hasse principle for three classes of varieties over global function fields.  Duke Math. J. {\bf 166} (2017), no. 17, 3349--3424. 

\bibitem[XW18]{XW18} Xiaoheng Wang, Maximal linear spaces contained in the base loci of pencils of quadrics.
Alg. Geom. {\bf 5} (3) (2018) 359--397.

\bibitem[Wi07]{Wi07} O. Wittenberg, Intersections de deux quadriques et pinceaux de courbes de genre 1,
Springer LNM {\bf 1901} (2007).

\bibitem[Wi15]{Wi15} O. Wittenberg, 
Rational points and zero-cycles on rationally connected varieties over number fields
in {\it Algebraic Geometry: Salt Lake City 2015}, Part 2, p. 597-- 635, Proceedings of Symposia in Pure Mathematics {\bf 97}, American Mathematical Society, Providence, RI, 2018.

 		

	\end{thebibliography}
\end{document}